\documentclass[11pt]{article}
\usepackage[ansinew]{inputenc}
\usepackage{graphicx}
\usepackage{color}
\usepackage{times}
\usepackage{enumerate,latexsym}
\usepackage{latexsym}
\usepackage{amsmath,amssymb}
\usepackage{graphicx}

\usepackage{amsthm}
\def\cC{\mathcal{C}}
\def\cL{\mathcal{L}}
\newfont{\bb}{msbm10 at 11pt}
\newfont{\bbsmall}{msbm8 at 8pt}
\def\cS{\mathcal{S}}

\def\S{{\Sigma}}

\newcommand{\ov}{\overline}
\def\rth{\mathbb{R}^3}
\def\R{\mathbb{R}}
\def\B{\mathbb{B}}
\def\N{\mathbb{N}}

\def\oB{\ov{\mathbb{B}}}

\def\be{\beta}

\newcommand{\inj}{\mbox{\rm Inj}}
\newcommand{\ben}{\begin{enumerate}}
\newcommand{\bit}{\begin{itemize}}
\newcommand{\een}{\end{enumerate}}
\newcommand{\eit}{\end{itemize}}
\newcommand{\wh}{\widehat}
\newcommand{\Int}{\mbox{Int}}
\newcommand{\cH}{{\mathcal H}}
\newcommand{\cT}{{\mathcal T}}

\newcommand{\ed}{\end{document}}

\def\a{{\alpha}}

\def\t{{\theta}}

\def\g{{\gamma}}
\def\G{{\Gamma}}

\def\de{{\delta}}

\def\ve{{\varepsilon}}

\newcommand{\cD}{{\mathcal D}}

\newtheorem{theorem}{Theorem}[section]

\newtheorem{proposition}[theorem]{Proposition}

\newtheorem{remark}[theorem]{Remark}
\newtheorem{corollary}[theorem]{Corollary}
\newtheorem{definition}[theorem]{Definition}

\newtheorem{claim}[theorem]{Claim}

\definecolor{b}{rgb}{.1,.1,.7}
\definecolor{rr}{rgb}{.8,0,.3}
\definecolor{g}{rgb}{0,.5,0}
\definecolor{pp}{rgb}{.5,0,.7}
\definecolor{r}{rgb}{.6,0,.3}
\definecolor{y}{rgb}{.9,.99,.9}

\begin{document}
\begin{title}
{Limit lamination theorems for H-surfaces}
\end{title}
\date{}
\begin{author}
{William H. Meeks III\thanks{This material is based upon
   work for the NSF under Award No. DMS-1309236.
   Any opinions, findings, and conclusions or recommendations
   expressed in this publication are those of the authors and do not
   necessarily reflect the views of the NSF.}
   \and Giuseppe Tinaglia\thanks{The second author was partially
   supported by
EPSRC grant no. EP/M024512/1}}
\end{author}
\maketitle
\vspace{-.3cm}

\begin{abstract}
In this paper we prove some general results for constant
mean curvature lamination limits of
certain sequences of compact  surfaces
 $M_n$ embedded in $\rth$ with constant mean curvature
 $H_n$ and fixed finite genus,
when the boundaries of these surfaces tend to infinity.
Two of these theorems generalize
to the non-zero constant mean curvature case,
similar structure theorems by Colding and Minicozzi
in~\cite{cm23,cm25} for limits of
sequences of minimal surfaces of fixed finite genus.
\end{abstract}

\vspace{.3cm}

\noindent{\it Mathematics Subject Classification:} Primary 53A10,
   Secondary 49Q05, 53C42

\noindent{\it Key words and phrases:}
Minimal surface, constant mean
curvature, one-sided curvature estimate,
curvature estimates, minimal lamination,
$H$-surface, $H$-lamination,
chord-arc, removable singularity, positive injectivity radius.
\maketitle

\section{Introduction}
In this paper  we apply results in~\cite{mt8,mt7,mt13,mt9} to obtain
(after passing to a subsequence)
constant mean curvature lamination limits for
sequences of compact
surfaces $M_n$ embedded in $\rth$ with constant mean
curvature $H_n$ and fixed finite genus,
when the boundaries of these surfaces tend to infinity in $\rth$.
These lamination limit results are inspired by and generalize
to the non-zero constant mean curvature setting
similar structure theorems by Colding and Minicozzi in~\cite{cm23,cm25}
in the case of
embedded  minimal surfaces; also see some closely related work of
Meeks, Perez and Ros in~\cite{mpr14, mpr11} in the minimal setting.

For clarity of  exposition, we will call an oriented surface
$M$ immersed in $\rth$ an {\it $H$-surface} if it
is {\it embedded}  and it
has {\it non-negative constant mean curvature $H$}.
In this manuscript  $\B(R)$ denotes the open ball in $\rth$
centered at the origin $\vec{0}$
of radius $R$ and for a
point $p$ on a surface $\Sigma$ in $ \rth$, $|A_{\Sigma}|(p)$ denotes the norm
of the second fundamental
form of $\Sigma$ at $p$.

\begin{definition} \label{def:lbsf} {\rm Let $U$ be an open set in $\rth$. \ben[1.]
\item We say that a sequence of smooth surfaces
$\Sigma(n)\subset U$  has {\em locally bounded norm of the second fundamental
form in $U$} if for every compact
subset $B$ in $U$, the norms of the second fundamental forms of the
surfaces $\Sigma(n)$
are uniformly bounded in $B$.
\item We say that a sequence of smooth surfaces
$\Sigma(n)\subset U$  has {\em locally positive injectivity radius
in $U$} if for every compact
subset $B$ in $U$, the injectivity radius functions  of the surfaces
$\Sigma(n)$ at points in $B$
are bounded away from zero for $n$ sufficiently large; see
Definition~\ref{definj} for the definition of the
injectivity radius function.
\item We  say that a sequence of smooth surfaces
$\Sigma(n)\subset U$  has {\em uniformly  positive
injectivity radius in $U$} if there exists
an $\ve>0$ such that for every compact
subset $B$ in $U$, the injectivity radius functions  of the
surfaces $\Sigma(n)$ at points in $B$
are bounded from below by $\ve$ for $n$ sufficiently large.
\een} \end{definition}

We will also need the next definition
 in the statement of Theorems~\ref{H-lam-thm} below,
as well as Definition~\ref{def:flux}  of the flux
of a 1-cycle in an $H$-surface, in the statements of
Theorems~\ref{H-lam-thm} and \ref{geometry2} below.

\begin{definition} {\em
A {\em strongly Alexandrov embedded} $H$-surface $f\colon \Sigma \to \rth$ is
a proper immersion of a complete surface $\Sigma$ of constant mean curvature
$H$ that extends to a proper
immersion of a complete three-manifold $W$ so that $\S$
is the mean convex boundary of $W$ and $f|_{\Int (W)}$ is injective.
See~\cite{mt3} for further discussion on this notion.}
\end{definition}

In this paper we wish to describe for any  large radius $R>0$,  the
geometry in  $\B(R)$ of any connected compact  $H$-surface $M$ in $\rth$ of
fixed finite genus that passes through the origin
and satisfies:  \ben \item the non-empty boundary of $M$
lies much farther than $R$  from the origin;
\item the injectivity radius function of $M$ is not
too small at points in $\B(R)$. \een
In order to obtain this geometric description  of $M$,
it is natural to consider a sequence
$\{M_n\}_{n\in\mathbb N}$
of compact $H_n$-surfaces in $\rth$
with finite genus at most $k$, $\vec{0}\in M_n$, $M_n$
contains no spherical components,
$\partial M_n\subset [\rth -\B(n)]$
and such that the sequence has locally
bounded injectivity radius  in $\rth$.  Then after passing to a subsequence and
possibly translating the surfaces $M_n$ by vectors of uniformly bounded length   so
that $\vec{0}\in M_n$ still holds, then exactly
one of the following three  possibilities occurs in the sequence:
\ben
\item $\{M_n\}_{n\in\mathbb N}$ has locally
bounded norm of the second fundamental form in $\rth$.
\item $\lim_{n\to \infty} |A_{M_n}|(\vec{0})=\infty$ and
$\lim_{n\to \infty} I_{M_n}(\vec{0})=\infty$.
\item $\lim_{n\to \infty} |A_{M_n}|(\vec{0})=\infty$ and
$\lim_{n\to \infty} I_{M_n}(\vec{0})=C$, for some $C>0$
\een
Depending on which of the above three mutually exclusive conditions
holds for $\{M_n\}_{n\in\mathbb N}$, one has
a limit geometric description given by its corresponding theorem listed below.

The next theorem corresponds to the case where
$\{M_n\}_{n\in\mathbb N}$ has locally
bounded norm of the second fundamental form in $\rth$.

\begin{theorem}\label{H-lam-thm}
Suppose that  $\{M_n\}_{n\in\mathbb N}$
is a sequence of compact $H_n$-surfaces in $\rth$
with finite genus at most $k$, $\vec{0}\in M_n$, $M_n$
contains no spherical components,
$\partial M_n\subset [\rth -\B(n)]$ and  the sequence has locally
bounded norm of the second fundamental form in $\rth$.
Then, after replacing  $\{M_n\}_{n\in\mathbb N}$  by a subsequence,
the sequence of surfaces $\{M_n\}_{n\in \mathbb N}$
converges with respect to the $C^\a$-norm, for any $\a\in(0,1)$, to a minimal lamination  $M_\infty$
of $\rth $ by parallel planes
or it converges smoothly (with
multiplicity  one or two) to a possibly
disconnected, strongly Alexandrov embedded
$H$-surface $M_\infty$ of genus at most $k$
and every component of $M_\infty$ is non-compact.
Moreover: \ben[1.]
\item  If the convergent sequence has uniformly positive injectivity
radius in $\rth$ or if $H=0$, then
the norm of the second fundamental form of $M_\infty$
is bounded.
\item If there exist positive numbers $I_0, H_0$
such that for $n$ large either the injectivity radius functions of the $M_n$
at $\vec{0} $ are bounded from above
by $ I_0$ or $H_n\geq H_0$, then the limit object
is a possibly
disconnected, strongly Alexandrov embedded
$H$-surface $M_\infty$ and
there exist a positive constant $\eta=\eta (M_\infty)$ and
simple closed oriented curves $\g_n \subset M_n$ with scalar
fluxes $F(\g_n)$ with  $\lim_{n\to\infty} F(\g_n)=\eta$.
\een
\end{theorem}

The next theorem corresponds to the case
where $\lim_{n\to \infty} |A_{M_n}|(\vec{0})=\infty$ and
$\lim_{n\to \infty} I_{M_n}(\vec{0})=\infty$.

\begin{theorem}\label{geometry1}
Suppose that  $\{M_n\}_{n\in\mathbb N}$
is a sequence of compact $H_n$-surfaces in $\rth$
with finite genus at most $k$, $\vec{0}\in M_n$, $M_n$
contains no spherical components,
$\partial M_n\subset [\rth -\B(n)]$, the sequence has
locally positive injectivity radius in $\rth$,
$\lim_{n\to \infty} |A_{M_n}|(\vec{0})=\infty$ and
$\lim_{n\to \infty} I_{M_n}(\vec{0})=\infty$.

Let $\cS\subset \rth$
denote the $x_3$-axis.
Then, after replacing   by a subsequence and applying a
fixed rotation that fixes the origin:  \ben[1.]
\item $\{M_n\}_{n\in\mathbb N}$ converges with respect to the $C^\a$-norm,
for any $\a\in(0,1)$,  to the
minimal foliation $\cL$ of
$\,\rth-\cS$ by horizontal planes punctured at points in $\cS$.
\item For any $R>0$ there exists $n_0\in \N$ such that for $n>n_0$,
there exists a possibly disconnected compact subdomain
$\cC_n$ of $M_n$, with $[M_n\cap \B(R/2)]\subset \cC_n \subset \B(R)$
and with $\partial \cC_n\subset \B(R)-\B(R/2)$,
consisting of a disk $\cD_n$ containing the origin and
possibly a second disk that intersects $\B(R/n)$, where
each disk has intrinsic diameter less than $3R$.

\item Away from  $\cS$,
each component of $ \cC_n$ consists of two multi-valued graphs spiraling
together to form a double spiral staircase (see Remark~\ref{remark:spiral} for an explicit
geometric description of the  double spiral staircase structure for \,$ \cC_n$). \een
\end{theorem}

The last theorem corresponds to the case
where $\lim_{n\to \infty} |A_{M_n}|(\vec{0})=\infty$ and
$\lim_{n\to \infty} I_{M_n}(\vec{0})=C$.

\begin{theorem}\label{geometry2}
Suppose that  $\{M_n\}_{n\in\mathbb N}$
is a sequence of compact $H_n$-surfaces in $\rth$
with finite genus at most $k$, $\vec{0}\in M_n$, $M_n$
contains no spherical components,
$\partial M_n\subset [\rth -\B(n)]$, the sequence has
locally positive injectivity radius in $\rth$,
$\lim_{n\to \infty} |A_{M_n}|(\vec{0})=\infty$ and
$\lim_{n\to \infty} I_{M_n}(\vec{0})=C$, for some $C>0$.

Let $\cS_0=\{(0,0,t)\mid t\in \R\}$,
$\cS_C=\{(C,0,t)\mid t\in \R\}$ and $\cS=\cS_0\cup \cS_C$.
Then, after replacing   by a subsequence and applying a
fixed rotation that fixes the origin: \ben[1.]
\item $\{M_n\}_{n\in\mathbb N}$ converges with respect to the $C^\a$-norm,
for any $\a\in(0,1)$,  to the
minimal foliation $\cL$ of
$\,\rth-\cS$ by horizontal planes punctured at points in $\cS$.
\item Given $R>C$ there exists $n_0\in\N$ such that for $n>n_0$, the
subdomain $\Delta_n$ of $M_n\cap \B(R)$
that intersects $\B(\frac R4)$ is a planar domain. In fact, $\Delta_n$
consists of a connected planar domain $\Delta_1(n)$ containing
the origin and possibly a
second connected planar domain $\Delta_2(n)$ and $\Delta_2(n)\cap\B(\frac Rn)\neq \O$.
Moreover, the intrinsic distance in $M_n$ between any two points in the same
connected component of $\Delta_n$ is less than $3R$.
Away from $\cS$, each  component of $\Delta_n$ consists of exactly two
multi-valued graphs spiraling
together.
Near $\cS_0$ and $\cS_C$, the pair of multi-valued graphs
form double spiral staircases
with opposite handedness  (see Remark~\ref{remark:spiral} for a
geometric description of each of the 1 or 2 components of $ \Delta_n$ near
points of $\cS$). Thus, circling only $\cS_0$ or only $\cS_C$
in $\Delta_n$ results in
going either up or down, while a path circling both $\cS_0$ and $\cS_C$ closes up.
\item There exist
simple closed oriented curves $\g_n \subset M_n$ converging
to the line segment joining the pair of
 points in $\cS\cap\{x_3=0\}$ and having lengths converging to
 $2C$ and fluxes converging to $(0,2C,0)$.
\een
\end{theorem}

In~\cite{mt7} we apply the non-zero flux
conclusions in Theorems~\ref{H-lam-thm} and \ref{geometry2} to obtain
curvature estimates away from the boundary
for any compact $1$-annulus in $\rth$ that has  scalar flux that is either
zero or greater than some $\rho>0$;
see Corollary~5.4 in~\cite{mt7} for this result.

The geometric description in item~2 of Theorem~\ref{geometry2} is
identical to the
geometric description of the $H=0$ case given in
Theorem~0.9 of   paper~\cite{cm25} by Colding and Minicozzi,
where in their situation the number of components in $\cC_n(R)$ must be one.
When the hypotheses of Theorem~\ref{geometry1} or \ref{geometry2}
hold, as $n$ approaches infinity
the convergent geometry of the surfaces $M_n$ around
the line or pair of lines in $\cS$
is that of a so-called ``parking garage structure".
See for instance~\cite{mpr14} for  the general notion and theory
of parking garage surfaces in $\rth$ and the notion of the
convergence of these surfaces to
a limit ``parking garage structure". This kind of
limiting structure and its application to
obtain curvature estimates for certain minimal planar
domains in $\rth$ first
appeared in work of Meeks, Perez and Ros in~\cite{mpr1}.

\vspace{.2cm}

\noindent  {\sc Acknowledgements:} The authors would
like to thank Joaquin Perez for
making Figure~\ref{fig2cone}.


\section{Preliminaries.} \label{sec:pre}

Throughout this paper, we use the following notation.
Given $a,b,R>0$, $p\in \rth$ and $\S$ a surface in $\rth$:

\bit
\item $\B(p,R)$ is the open ball of radius $R$ centered at $p$.
\item $\B(R)=\B(\vec{0},R)$, where $\vec{0}=(0,0,0)$.
\item For $p\in \S$, $B_{\S}(p,R)$ denotes the open
intrinsic ball in $\S$ of radius $R$.
\item $A(r_1,r_2)=\{(x_1,x_2,0)\mid r_2^2\leq x_1^2+x_2^2\leq r_1^2\}$.
\eit

We first  introduce the notion of multi-valued graph,
see~\cite{cm22} for further discussion.
Intuitively, an $N$-valued graph is a simply-connected
embedded surface covering an
annulus such that over a neighborhood of each point of the annulus, the
surface consists of $N$ graphs. The stereotypical  infinite multi-valued
graph is half of the helicoid, i.e., half of an infinite
double-spiral staircase.

\begin{figure}[h]
\begin{center}
\includegraphics[width=11.5cm]{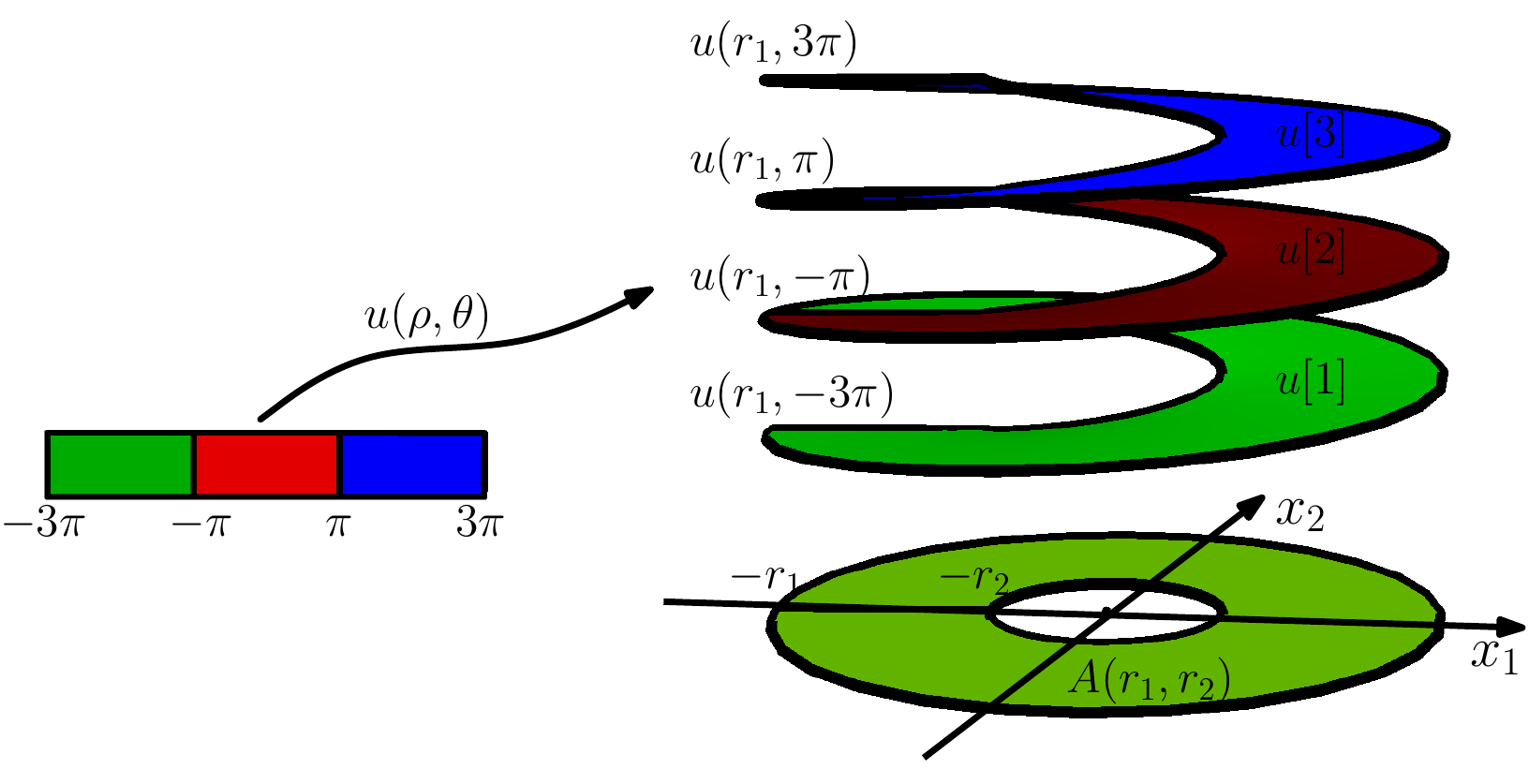}
\caption{A right-handed 3-valued graph.}
\label{3-valuedgraph}
\end{center}
\end{figure}

\begin{definition}[Multi-valued graph]\label{multigraph} {\rm
Let $\mathcal{P}$ denote the universal cover of the
punctured $(x_1,x_2)$-plane,
$\{(x_1,x_2,0)\mid (x_1,x_2)\neq (0,0)\}$, with global coordinates
$(\rho , \theta)$.
\ben[1.] \item
An {\em $N$-valued graph over the annulus $ A(r_1,r_2)$} is a single valued graph
$u(\rho, \theta)$ over $\{(\rho ,\theta )\mid r_2\leq \rho
\leq r_1,\;|\theta |\leq N\pi \}\subset \mathcal{P}$, if $N$ is odd,
or over $\{(\rho ,\theta )\mid r_2\leq \rho
\leq r_1,\;(-N+1)\pi\leq \theta \leq \pi (N+1)\}\subset \mathcal{P}$, if $N$ is even.
\item An  $N$-valued graph $u(\rho,\t)$ over the annulus $ A(r_1,r_2)$ is
called {\em righthanded} \, [{\em lefthanded}] if whenever it makes sense,
$u(\rho,\t)<u(\rho,\t +2\pi)$ \, [$u(\rho,\t)>u(\rho,\t +2\pi)$]
\item The set $\{(r_2,\theta, u(r_2,\theta)), \theta\in[-N\pi,N\pi]\}$ when $N$ is odd
(or $\{(r_2,\theta, u(r_2,\theta)), \theta\in[(-N+1)\pi,(N+1)\pi]\}$ when $N$ is even)
is the {\em inner boundary} of the $N$-valued graph.
\een }
\end{definition}

From Theorem 2.23 in~\cite{mt7} one obtains the following,
detailed geometric description
of an $H$-disk with large
norm of the second fundamental form at the origin. The precise meanings
of certain statements below are
made clear in~\cite{mt7} and we refer the reader to that paper for further details.

\begin{theorem}\label{mainextension}
Given $\ve,\tau>0$   and $\ov{\ve}\in (0,\ve/4)$ there
exist  constants $\Omega_\tau:=\Omega(\tau )$,
$\omega_\tau:=\omega(\tau )$ and  $G_\tau:=G(\ve,\tau,\ov\ve ) $
such that if $M$ is an $H$-disk, $H\in (0,\frac 1{2\ve})$,
$\partial M\subset \partial \B(\ve)$, $\vec 0\in M$ and
$|A_M|(\vec 0)>\frac 1\eta G_\tau$, for $\eta\in (0,1]$, then for
any $p\in \B(\vec{0},\eta\ov{\ve})$ that is a maximum of the
function
$|A_{M}|(\cdot)(\eta\bar\ve-|\cdot|)$, after translating
$M$ by $-p$, the following geometric description of $M$ holds:
\par

\begin{itemize}

\item On the scale of the norm of the second
fundamental form  $M$ looks like one or two helicoids nearby  the
origin and, after a rotation that turns these helicoids into
vertical helicoids, $M$  contains a 3-valued graph
$u$ over  $A(\ve\slash\Omega_\tau,\frac{\omega_\tau}{|A_M|(\vec 0)})$
with norm of its gradient less than $\tau$ and with inner boundary
in $\B(10\frac{\omega_\tau}{|A_M|(\vec 0)})$.

\item Moreover, given $j\in\mathbb N$ if we let the constant $G_\tau$ depend on $j$
as well, then $M$ contains $j$ pairwise disjoint
3-valued graphs with their inner boundaries in
$\B(10\frac{\omega_\tau}{|A_M|(\vec 0)})$.
\end{itemize}

 \end{theorem}

  Theorem~\ref{mainextension}
 was inspired by the pioneering work of Colding and Minicozzi  in the
 minimal case~\cite{cm21,cm22,cm24,cm23};
however in the constant positive mean curvature setting this
description has led to a different conclusion, that is the existence
of the intrinsic curvature estimates stated
below.

\begin{theorem}[Intrinsic curvature estimates, Theorem~1.3 in~\cite{mt7}] \label{cest}
Given $\delta,\cH>0$, there exists a constant $K (\delta,\cH)$ such that
for any $H$-disk $\cD$ with $H\geq \cH$,
{ $${{\sup}_{ \{p\in \cD \, \mid \, d_{\cD}(p,\partial
\cD)\geq \delta\}} |A_\cD|\leq  K (\delta,\cH)}.$$}
\end{theorem}

\vspace{.1cm}

Rescalings of a helicoid
give a sequence of embedded minimal disks with arbitrarily
large norm of the second fundamental
form at points arbitrarily far
from its boundary; therefore in the minimal setting, similar
curvature estimates do not hold.

The next two results from~\cite{mt9} will also be
essential tools that we use in this paper.

\begin{theorem}[Extrinsic one-sided curvature estimates for $H$-disks] \label{th}
There exist $\ve\in(0,\frac{1}{2})$ and
$C \geq 2 \sqrt{2}$ such that for any $R>0$, the following holds.
Let $\cD$ be an $H$-disk such that $$\cD\cap \B(R)\cap\{x_3=0\}
=\O \quad \mbox{and} \quad \partial \cD\cap \B(R)\cap\{x_3>0\}=\O.$$
Then:
\begin{equation} \label{eq1}
\sup _{x\in \cD\cap \B(\ve R)\cap\{x_3>0\}} |A_{\cD}|(x)\leq \frac{C}{R}.
\end{equation} In particular, if
$\cD\cap \B(\ve R)\cap\{x_3>0\}\neq\O$, then $H\leq \frac{C}{R}$.
\end{theorem}

The next corollary follows immediately from   Theorem~\ref{th} by a simple
rescaling argument. It roughly states that
we can replace the  $(x_1,x_2)$-plane   by any surface that has a fixed
uniform estimate on the norm of its second fundamental form.

\begin{corollary} \label{cest2}
Given an $a\geq 0$,
there exist $\ve\in(0,\frac{1}{2})$ and $C_{a} >0$ such that
for any $R>0$, the following holds.
Let $\Delta$ be a compact immersed surface in $\B(R)$ with
$\partial \Delta \subset \partial \B(R)$, $\vec{0}\in \Delta$
and satisfying $|A_{\Delta}| \leq a/R$. Let $\cD$ be an $H$-disk such that
$$\cD\cap \B(R)\cap\Delta=\O \quad \mbox{and} \quad \partial \cD\cap \B(R)=\O.$$
Then:
\begin{equation} \label{eq1*}
\sup _{x\in \cD\cap \B(\ve R)} |A_{\cD}|(x)\leq \frac{C_{a}}{R}.
\end{equation} In particular, if $\cD\cap \B(\ve R)\neq \O$, then $H\leq \frac{C_{a}}{R}$.
\end{corollary}

The next curvature estimate is a more involved application of Theorem~\ref{th} and
also uses Theorem~\ref{thm2.1} below in its proof.

\begin{corollary}[Corollary 4.6  in~\cite{mt13}] \label{cest-cor}
There exist constants $\ve<1$, $C>1$
such that the following holds.  Let
$\Sigma_1$, $\Sigma_2$, $\Sigma_3$ be three pairwise disjoint $H_i$-disks with
$\partial \Sigma_i\subset [ \rth- \B(1)]$ \,for $i=1,2,3$.
If $\,\B(\ve)\cap\Sigma_i\not=\O$ for $i=1,2,3$, then {
\[
\sup_{\B(\ve)\cap\Sigma_i,\,i=1,2,3}|A_{\Sigma_i}|\leq C.
\]}\end{corollary}
\vspace{.1cm}

In~\cite{mt8}, we applied the one-sided curvature estimates in Theorem~\ref{th}
to prove a relation between intrinsic
and extrinsic distances in an $H$-disk, which can be viewed as a
{\em weak chord arc} property.
This result was motivated by and generalizes a previous result by
Colding-Minicozzi for 0-disks, namely
Proposition~1.1 in~\cite{cm35}.
We begin by making the following definition.

\begin{definition} Given a point $p$ on a surface
$\Sigma\subset \rth$, $\S (p,R)$ denotes the closure of
 the component of $\Sigma \cap {\B}(p,R)$ passing through $p$.
\end{definition}

\begin{theorem}[Weak chord arc property, Theorem 1.2 in~\cite{mt8}] \label{thm1.1}
There exists a $\delta_1 \in (0,
\frac{1}{2})$  such that the following holds.

Let $\S$ be an   $H$-disk in $\rth.$  Then for all
intrinsic closed balls $\ov{B}_\S(x,R)$ in $\S-
\partial \S$:

\ben \item $\S (x,\delta_1 R)$ is a disk with piecewise smooth boundary
$\partial \Sigma(x,\delta_1 R)\subset \partial \B(\de_1R)$. \item
$
 \S (x, \delta_1 R) \subset B_\S (x, \frac{R}{2}).$
\een
\end{theorem}

For applications here, we will also need the closely related chord-arc result below, that is Theorem 1.2 in~\cite{mt13}.

\begin{theorem}[Chord arc property  for $H$-disks] \label{main2}
There exists a constant $a>1$ so that the following holds.
Suppose that $\S$ is an $H$-disk with $ \vec{0}\in\S$, $R>r_0>0$ and
${B}_\S(\vec 0,aR) \subset \S-\partial \S$. If
$\sup_{\large B_\S(\vec 0,(1-\frac{\sqrt{2}}{2})r_0)}|A_\S|>r_0^{-1}$,
then
$$ \frac{1}{3}\mbox{\rm dist}_\S (x,\vec{0})\leq |x|/2+r_0, \;
\mbox{\rm for } x\in B_\S(\vec 0,R).$$
\end{theorem}

Since in the proofs of
Theorems~\ref{H-lam-thm}, \ref{geometry1} and \ref{geometry2} we will frequently
refer to parts of the statement of the Limit Lamination Theorem for $H$-disks, namely
Theorem~1.1 in~\cite{mt13},
we state it below for more direct referencing.

\begin{theorem}[Limit lamination theorem for $H$-disks] \label{thm2.1}
Fix  $\ve >0$ and let $\{M_n\}_n$ be
a sequence of $H_n$-disks in $\rth$ containing the origin  and such that
$\partial M_n \subset [\rth - \B(n)]$ and
$|A_{M_n} |(\vec{0})\geq \ve$. Then, after replacing  by
some subsequence,  exactly one of the
following two statements hold.
\ben[A.]
\item The surfaces $M_n$
converge smoothly with multiplicity one or two on compact
subsets of $\rth$ to a helicoid $M_{\infty}$
containing the origin. Furthermore,
every component $\Delta$ of $M_n\cap \B(1)$  is an open disk
whose closure $\ov{\Delta}$ in $M_n$
is a compact disk with piecewise smooth boundary, and
where the intrinsic distance in $M_n$
between  any two points of its closure $\ov{\Delta}$ less than 10.
\item  There are points $p_n\in M_n$ such that
\[
\lim_{n\to \infty}p_n=\vec{0} \text{\, and \,}
\lim_{n\to \infty}|A_{M_n}|(p_n)=\infty,
\] and the
following  hold: \ben
\item The surfaces $M_n$ converge to a foliation
of $\rth$ by planes and the convergence is $C^\alpha$, for any $\a\in(0,1)$,
away from the line containing
the origin and orthogonal to the planes in the foliation. \item
There exists  compact subdomains
$\cC_n$ of $M_n$, $[M_n\cap \oB(1)]\subset \cC_n \subset \B(2)$
and  $\partial \cC_n\subset \B(2)-\oB(1)$, each $\cC_n$
consisting of one or two pairwise disjoint disks, where
each disk component has intrinsic diameter less than 3 and intersects $\B(1/n)$.
Moreover, each connected
component of $M_n\cap \B(1)$ is an open disk whose closure in $M_n$
is a compact disk with piecewise smooth boundary. \een
\een
\end{theorem}

\begin{remark}[Double spiral  staircase structure] \label{remark:spiral} {\em Suppose that Case B
occurs in the statement of Theorem~\ref{thm2.1}
and let  $\Delta_n$ be a component of $\cC_n$.
By Remark~3.6 in~\cite{mt13}, after replacing  the surfaces $M_n$ by a subsequence and
composing them by a  rotation of $\rth$ that fixes the origin and so that the planes of
the limit foliation are horizontal, then, as $n$ tends to infinity, $\Delta_n$ has the structure of
a {\em double spiral  staircase}, in the following sense:
\ben \item $\Delta_n $ contains a smooth connected arc $\G_n(t)$,
called its {\em central column}, that is
parameterized by the set of its third coordinates which equals the
interval  $I_n=(-1-\frac1n,1+\frac1n)$. $\G_n(t)$ is
 the set of points of $\Delta_n $ with vertical tangent planes
 and $\G_n(t)$ is $\frac1n$-close to the arc
$\{(0,0,t) \mid t\in I_n\}$ with respect to the $C^1$-norm.

For each $t\in I_n$,  let $T_n(t)$ be the vertical tangent plane
of $\Delta_n$ at $\G_n(t)$.
\item  For $t\in I_n$, $T_n(t)\cap \Delta_n$ contains a smooth
arc $\a_{n,t}$ passing through $\G_n(t)$ that is $\frac1n$-close
in the $C^1$-norm to an arc $\beta_{n,t}$ of the line $T_n(t)\cap \{x_3=t\}$
such that $\G_n(t)\in \beta_{n,t}$ and the end points of  $\beta_{n,t}$ lie in $\B(2)-\ov{\B}(1)$; here
$\{\a_{n,t}\}_{t\in   I_n}$ is a pairwise disjoint collection of arcs and $\Delta_n=\bigcup_{t\in I_n}\a_{n,t}$.

\item The absolute Gaussian curvature of $\Delta_n$ along $\G_n(t)$ is pointwise
greater than $n$. Since the central column $\G_n(t)$ of $\Delta_n$ is converging $C^1$ to the segment given by
$\B(1)\cap \{x_3\text{-axis}\}$, the arcs $\a_{n,t}$ are converging to $T_n(t)\cap \Delta_n$ and on the scale of curvature
$\Delta_n$ is closely approximated by a vertical helicoid near  every point of $\G_n(t)$ (see Corollary~3.8 in~\cite{mt9}),
then the rate of change of  the horizontal unit normal of $T_n(t)$ along $\G_n(t)$ is greater than $\sqrt{n}$.  \een
}
\end{remark}

Next, we recall the notion of
flux of a 1-cycle of an $H$-surface; see for instance~\cite{kks1,ku2,smyt1}
for further discussions of this invariant.

\begin{definition} \label{def:flux} {\rm
Let $\gamma$ be a 1-cycle in an $H$-surface $M$. The 
{\em flux} of
$\gamma$ is $F(\g)=\int_{\gamma}(H\gamma+\xi)\times \dot{\gamma}$, where $\xi$
is the unit normal to $M$ along $\gamma$. The norm
$|F(\g)|$ is called the {\em scalar flux} of $\g$.}
\end{definition}

The flux of a 1-cycle in   an $H$-surface $M$ is a homological invariant and
we say that  $M$ has {\em  zero flux} if the flux of any 1-cycle in $M$ is zero;
in particular, since the first homology group of a disk is zero,   an $H$-disk has  zero flux.
Finally, the next definition was needed in Definition~\ref{def:lbsf} in the Introduction.

\begin{definition} \label{definj} {\rm
The injectivity radius $I_M(p)$ at a point
$p$ of a complete Riemannian manifold $M$ is
the supremum of the radii $r>0$ of the open metric balls $B_M(p,r)$ for which
the exponential map at $p$ is a diffeomorphism. This defines
the {\it injectivity radius function,}
$I_M\colon M\to (0,\infty ]$, which is continuous
on $M$ (see e.g., Proposition~88 in~\cite{ber1}). When $M$ is
complete, we let $\inj(M)$ denote the
{\it injectivity radius of $M$}, which is defined to be the
infimum of $I_M$.}
\end{definition}

\section{The proof of Theorem~\ref{H-lam-thm}.} \label{sec:finiteg}

In this section we will prove all of the statements in
Theorem~\ref{H-lam-thm} except for one of the implications
in item~2. However, at the end of this section we explain
how the  missing proof of this implication
follows from
 item~2 of Theorem~\ref{geometry2}.  Hence, once
 Theorem~\ref{geometry2} is proven in Section~\ref{sec5}, the proof of
 Theorem~\ref{H-lam-thm} will be complete.

 Let
 $\{M_n\}_{n\in\mathbb N}$
be a sequence of compact $H_n$-surfaces in $\rth$
with finite genus at most $k$, $\vec{0}\in M_n$, $M_n$
contains no spherical components,
$\partial M_n\subset [\rth -\B(n)]$ and  the sequence has locally
bounded norm of the second fundamental form in $\rth$.
By a standard argument,
a subsequence of the surfaces converges to a weak $H$-lamination
$\cL$ of   $\rth$; see  the
references
\cite{mpr10,mpr18,mt4} for this argument and the Appendix for the
definition and some key properties of a weak
$H$-lamination that we will apply below.

Let $L$ be a leaf of $\cL$. If $L$ is stable ($L$ admits a positive
Jacobi function), then $L$ is a complete,
stable constant mean curvature surface in $\rth$,
which must be a flat plane by~\cite{lor2,ros9}.  If $L$ is a
flat plane, then the injectivity radius is infinite.
Since by Theorem~4.3 in~\cite{mpr19} limit leaves
(see Definition~\ref{deflimit} for the
definition of limit leaf) of $\cL$ are stable,
we conclude that if $L$ is a limit leaf of $\cL$, then it is a plane and has
infinite injectivity radius.  Thus we also conclude that if
$\cL$ has a limit leaf, then $H=0$. From this point till the
beginning of the proof of item~1 of the theorem,
we will assume that $\cL$ is not a lamination of $\rth$ by parallel planes.

Suppose now that  $L$ is a non-flat leaf of
$\cL$.  By the discussion in the previous paragraph
and item~3 of Remark~\ref{remarkweak},
$L$ is a non-limit leaf
and it has on its mean convex side an embedded half-open regular neighborhood
$N(L)$ in $\rth$ that intersects $\cL$ only in the leaf $L$; also since
$L$ is not a limit leaf of $\cL$, then $N(L)$ lies in the interior of an
open set $\wh{N}(L)$ that also intersects $\cL$ only in the leaf $L$.
Since the leaf $L$
is not stable, the existence of  $N(L)$ allows us to apply
the arguments in the proof of Case A in the proof of Proposition~3.1 of~\cite{mt9},
 to show that
the sequence $\{M_n\cap \wh{N}(L)\}_{n\in \N}$ converges to $L$ with multiplicity one or  two
and the genus of $L$ is at most $k$.

We now
prove that $\cL$ does not contain a limit leaf.
Arguing by contradiction, suppose $L$ is a limit leaf of $\cL$,
then, as previously proved, $L$ must be a flat plane. Thus, since
we are assuming that $\cL$ is not a lamination of $\rth$ by parallel planes,
 $\cL$ is a minimal lamination containing a flat leaf $L$ and a non-flat
leaf $L'$ with finite genus at most $k$.
By Theorem~7 in~\cite{mr13}, a finite genus
leaf of a minimal lamination of $\rth$  is proper,
which contradicts the Half-space Theorem~\cite{hm10} since $L'$ is contained
in the half-space determined by $L$. This proves  that
$\cL$ contains no limit leaves.

Since  $\cL$ is a weak $H$-lamination of $\rth$ that does not have a limit leaf,
then the union of the leaves of $\cL$ is a properly
immersed, possibly disconnected
 $H$-surface, such that around any point $p$ where the leaves of
 the weak lamination do not form a lamination, there
exists an $\ve>0$ such that $\cL\cap \B(p,\ve)$ consists
of exactly two disks in leaves of $\cL$ with boundaries in
$ \B(p,\ve)$ and these two disks
lie on one side of each other, intersect at $p$ and their non-zero
mean curvature vectors are oppositely oriented.  See the Appendix for further discussion
of properties of weak $H$-laminations.

 If $H=0$, then the leaves of $\cL$ are  embedded by the maximum
 principle and $\cL$ is connected because of the Strong Halfspace Theorem~\cite{hm10}. Thus
by elementary separation properties, $\cL$  bounds a proper region
$W$ of $\rth$. Hence, $\cL$  is a connected, strongly
Alexandrov embedded minimal surface in the case where $H=0$.

Suppose next that $H>0$ and note that by the previous description
or by item~3 in Remark~\ref{remarkweak},
each leaf $L$ of $\cL$ can be perturbed  slightly on its mean convex
side to be properly embedded and hence $L$ is strongly Alexandrov embedded.  By
Theorem~2 in~\cite{enr1}, for any two components $\Sigma_1$ and
$\Sigma_2$ of $\cL$,  $\Sigma_1$ does not
lie in the mean convex component of $\rth-\Sigma_2$.
It follows that each of   the components of $\rth-\cL$, except for
one, is a mean convex domain
with one boundary component. This means  $\cL$
corresponds to a possibly disconnected strongly Alexandrov embedded
$H$-surface.
Finally, since closed Alexandrov embedded $H$-surfaces in $\rth$
are round spheres and no component of $M_n$ is spherical,
a monodromy argument implies that each leaf of the limit lamination
$\cL$ is non-compact. Setting $M_\infty:=\cL$ finishes the
proof of the first statement of the
theorem.

We next prove item~1 in the theorem.  Namely, we will prove that if the sequence
$\{M_n\}_{n\in\mathbb N}$ has uniformly positive injectivity radius
in $\rth$ or if $H=0$, then
the norm of the second fundamental form of $M_\infty$
is bounded. If the constant mean curvature of
$M_\infty$ is positive and if the sequence $\{M_n\}_{n\in\mathbb N}$ has
uniformly positive injectivity
radius in $\rth$, then the norms of the second fundamental
forms of the surfaces  $M_n$ converging to $M_\infty$ on any compact region of
$\rth$ are eventually bounded from above by a constant that
only depends on the curvature estimate given in Theorem~\ref{cest}; hence $M_\infty$
has uniformly bounded norm of its second fundamental form in this case.  If the
mean curvature of $M_\infty$ is zero, then as observed already
either $M_\infty$ is a lamination of $\rth$ by parallel planes
or else $M_\infty$ is a properly embedded connected
minimal surface in $\rth$ of finite genus.  If $M_\infty$ is a
lamination of $\rth$ by parallel planes then the claim
is clearly true. Otherwise, by the classification of the asymptotic
behavior of properly embedded minimal surfaces in $\rth$ with finite genus given in
the papers~\cite{bb2,col1,mpr6},
the norm of the second fundamental form of the unique leaf of
$M_\infty$ is also bounded in this case.
This last observation completes the proof of  item~1 in the theorem.

We next consider the proof of  item~2 in the theorem.
Namely, we will prove that if there
exist positive numbers $I_0, H_0$
such that for $n$ large either the injectivity radius functions of the surfaces $M_n$
at $\vec{0} $ are bounded from above by $ I_0$ or $H_n\geq H_0$, then $M_\infty$
is a strongly Alexandrov embedded
$H$-surface  and
there exist a positive constant $\eta=\eta (M_\infty)$ and
simple closed oriented curves $\g_n \subset M_n$ with scalar
fluxes $F(\g_n)$ with $\lim_{n\to\infty} F(\g_n)=\eta$.

We first show that $M_\infty$ cannot be
a lamination of $\rth$ by parallel planes.
Arguing by contradiction, suppose that the sequence
$\{M_n\}_{n\in\mathbb N}$ converges to a lamination of $\rth$ by parallel planes.
In particular  $\lim_{n\to\infty}H_n=0$ and
the injectivity radius functions of the surfaces $M_n$
at $\vec{0} $ are bounded from above by $I_0$.  Then,
 for $n$ large, the Gauss equation implies that
 the $\limsup K_{M_n}$  of the Gaussian curvature functions
 of the surfaces $M_n$ is non-positive.
 Classical results on Jacobi fields along geodesics
in such surfaces imply that for $n$ large the exponential map of $M_n$ at $\vec 0$
 on the  closed disk in $T_{\vec 0}M_n$ of a certain radius $r_n\in (0, I_0]$
is a local diffeomorphism that is injective on the interior of
the disk but it is not injective along its boundary
circle of radius $r_n$ (see, for instance, Proposition~2.12, Chapter~13 of~\cite{doc2}).
Moreover, since the sequence $\{M_n\}_{n\in \N}$ has
locally bounded norm of the second fundamental form, the sequence
of numbers $r_n$ is bounded away from zero. Hence,
there exists a sequence of simple closed geodesic loops
$\a_n\subset M_n$ based at $\vec 0$ and of lengths uniformly
bounded from below and above
that are smooth everywhere
except possibly at $\vec 0$. By the nature of the convergence,
$\a_n$ converges to a geodesic loop in $M_\infty$ based at $\vec 0$.
Therefore $M_\infty$ cannot be
a lamination of $\rth$ by parallel planes.

We next prove
the existence of the 1-cycles $\g_n\subset M_n$ with non-zero flux
described in item~2 of the theorem.
Since $M_\infty$ cannot be
a lamination of $\rth$ by parallel planes, by the already proved first
main  statement of the theorem,
the sequence  $\{M_n\}_{n\in\mathbb N}$
converges with multiplicity one or two to a possibly
disconnected, non-flat strongly Alexandrov embedded
$H$-surface $M_\infty$ of genus at most $k$. Since the convergence to
$M_\infty$ is with  multiplicity one or two, a curve lifting argument
shows that in order to prove that item~2 holds,
it suffices to show that $M_\infty$ has non-zero flux.

If $\lim_{n\to\infty} H_n=0$ but
the injectivity radius functions of the $M_n$
at $\vec{0} $ are bounded from above by $I_0$, then the same
arguments as before imply that $M_\infty$ is not simply-connected
because a simply-connected minimal surface cannot contain a
geodesic loop. Thus,
by the results in~\cite{mt2}, the finite genus minimal surface $M_\infty$ must have
non-zero flux.

It remains to
consider the case that $H_n\geq H_0>0$. In this case  $M_\infty$ is
a  proper collection of
non-zero constant mean curvature surfaces,
each component of which is non-compact and the entire surface has finite genus
at most $k$. Abusing the notation, let $M_\infty$
denote the component containing the origin. If $M_\infty$ has injectivity radius
function uniformly bounded from below by a positive constant,
then it has uniformly bounded norm of the
second fundamental form by Theorem~\ref{cest} and again,  by the results
in~\cite{mt2}, $M_\infty$ has non-zero flux.

In other words,  item~2 can only fail if  $M_\infty$ has positive mean
curvature but the injectivity radius function is not bounded from below
 by a positive constant.
In this case we can  apply the blow-up argument described in Proposition~\ref{cor:5.7}.
Such a blow-up argument gives the following.  Let $p_n\in M_\infty$ be a sequence of points such
that $\lim_{n\to\infty}{I_{M_\infty}(p_n)}= 0$ and let $q_n$ be a sequence of points with
almost-minimal injectivity radius for $B_{M_\infty}(p_n,1)$, see Definition~\ref{amininj}.
Then, by Proposition~\ref{cor:5.7} there exist
 positive numbers $R_n$, $\lim_{n\to\infty}R_n= \infty$, such that after replacing by a subsequence the component
${\bf M}_n$ of $\frac{1}{I_{M_\infty}(q_n)}[M_\infty-q_n]\cap \B(R_n)$
 containing $\vec 0$ has boundary in $\partial\B(R_n)$ and the following properties hold:
\bit
\item ${\bf M}_n$ has finite genus at most $k$.
\item $I_{{\bf M}_n}(x)\geq 1\slash 2$ for any $x\in {\bf M}_n\cap \B(R_n\slash 2)$
and $I_{{\bf M}_n}(\vec 0)=1$. \item The mean curvatures ${\bf H}_n$ of the ${\bf M}_n$
converge to zero as $n $ goes to infinity.\eit

Suppose for the moment that  the sequence  ${\bf M}_n$  has locally bounded norm of the second fundamental form
in $\rth$.
By the already proven main statement of Theorem~\ref{H-lam-thm} applied to the sequence  ${\bf M}_n$,
a subsequence converges to a possibly
disconnected, strongly Alexandrov embedded
$H$-surface ${\bf M}_\infty$ of genus at most $k$
and every component of $M_\infty$ is non-compact.
Since $\lim_{n\to\infty} {\bf H}_n=0$ but
the injectivity radius functions of the ${\bf M}_n$
at $\vec{0} $ are bounded from above by $1$, then our previous
arguments imply that  the finite genus minimal surface ${\bf M}_\infty$ must have
non-zero flux.
Hence, we may assume that the sequence ${\bf M}_n$ fails to have locally bounded norm of the second
fundamental form. Assume for the moment that  Theorem~\ref{geometry2} holds. In
the case that we are considering, we can apply item~2 of  Theorem~\ref{geometry2} to conclude that
the surfaces ${\bf M}_n$ have non-zero flux,
which would mean that $M_\infty$ has non-zero flux as well. The construction of the closed
curves, called connection loops,
with non-zero flux is described in detail after Remark~\ref{moneortwo}.

In summary, the proof of
 the Theorem~\ref{H-lam-thm} will be complete once   Theorem~\ref{geometry2} is proven
 in Section~\ref{sec5}.

\section{The proof of Theorem~\ref{geometry1}.} \label{sec4}
In this section we will prove  Theorem~\ref{geometry1}.
Suppose that  $\{M_n\}_{n\in \mathbb N}$
is a sequence of compact $H_n$-surfaces in $\rth$
with finite genus at most $k$, $\vec{0}\in {M_{n}}$, ${M_{n}}$
contains no spherical components,
$\partial {M_{n}}\subset [\rth -\B(n)]$, the sequence has
locally positive injectivity radius in $\rth$ and
$\lim_{n\to \infty} |A_{{M_{n}}}|(\vec{0})=\infty$.
Since we will use some of the results proved here in the proof
of Theorem~\ref{geometry2}, we will for the moment not invoke the additional
hypothesis $\lim_{n\to \infty} I_{{M_{n}}}(\vec{0})=\infty$.

Since $\lim_{n\to \infty} |A_{{M_{n}}}|(\vec{0})=\infty$
and $I_{M_n}(\vec{0})$ is bounded from below by some positive
number, then Theorem~\ref{cest}
implies that $\lim_{n\to \infty} H_n =0$.
After replacing by a subsequence, there exists a smallest
closed nonempty set $\cS\subset \rth$
such that the sequence $\{M\}_{n\in \mathbb N}$ has locally bounded norm of
the second fundamental form in $\rth-\cS$ and
converges with respect to the $C^\a$-norm, for any $\a\in (0,1)$,  to a
nonempty minimal lamination $\cL$ of
$\rth-\cS$; the set $\cS$ is smallest in the sense that every
subsequence fails to converge to a minimal lamination
in a proper subset of $\cS$.
The proofs of the existence of $\cS$ and $\cL$ are the
same as those appearing in the proofs of the
first three items in Claim~3.4 in~\cite{mt9} and we
refer the reader to~\cite{mt9} for the details.

We begin by studying the geometry of $\cL$ and $\cS$.
The local analysis presented here
 is analogous to and inspired by the one given in the minimal case  considered by
 Colding and Minicozzi in~\cite{cm25}. Indeed the structure of the lamination
 nearby points in $\cS$ is identical.

Let $p\in \cS$. After replacing by a subsequence,
there exists a sequence of points $p_n\in {M_{n}}$ converging to $p$
such that the norm of the second fundamental form of
${M_{n}}$ at $p_n$ is at least $n$. Since we may assume that the
injectivity radius function of
$M_n$ is at least some $\ve>0$ at  $p_n$, then applying  Theorem~\ref{thm1.1},
we find that
for $n$ large,
the intersection of each $H_n$-disk $B_{{M_{n}}}(p_n,\ve)$  with
$\B(p,  \de_1\ve)$
contains a component ${M_{n}}(p_n,  \de_1\ve)$ that is an $H_n$-disk
with boundary in the boundary
of $\B(p_n,  \de_1\ve)$.
Theorem~\ref{mainextension} now gives that for $n$ large,
there exists a collection of 3-valued
graphs $\{G_1(n),\ldots,G_{k(n)}(n)\}$ with inner boundaries
converging to $p$, norms of their gradients   at most one and
$\lim_{n\to \infty} k(n)= \infty$. Since $\{G_1(n),\ldots,G_{k(n)}(n)\}$ is a
collection of embedded and pair-wise disjoint 3-valued graphs contained in a compact ball,
it must contain a sequence of 3-valued graphs for which the distance between the
sheets is going to zero. Hence, after reindexing, we can assume that
the 3-valued graphs $G_1(n)$ are collapsing in
the limit to a minimal disk $D(p)\subset \B(p,s)$ of gradient at most
one  over its tangent plane at $p$ and where
$\partial D(p)\subset \partial \B(p,s)$ and
$s<\delta_1\ve$ is  fixed and depending on $\ve$;  actually one produces
the punctured graphical disk
\[
D(p,*)=D(p)-\{p\}
\]
 as a limit
and then $p$ is seen to be a removable singularity.
\begin{figure}
\begin{center}
\includegraphics[width=7cm]{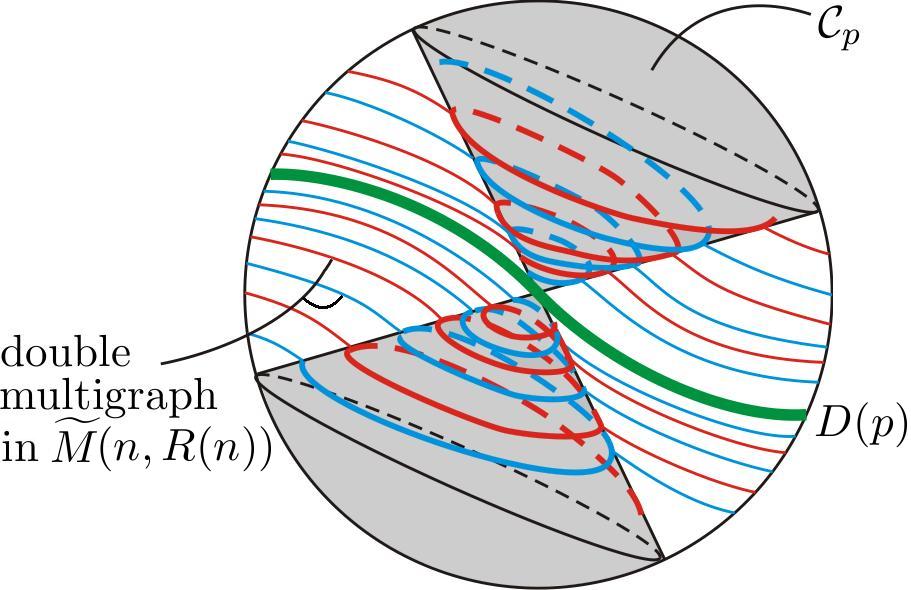}
\caption{The local {\em Colding-Minicozzi picture}  around a
point where the curvature blows up. The stable punctured disk
$D(p,*)$ appears in the limit lamination.}
\label{fig2cone}
\end{center}
\end{figure}

By the  one-sided curvature estimate in Corollary~\ref{cest2},
the 3-valued graph $G_1(n)$
gives rise to curvature estimates at points of ${M_{n}}$ nearby $G_1(n)$
and, as $n$ goes to infinity, these curvature estimates give rise to
curvature estimates in compact subsets of the complement set
\[
W(p)= \B(p,s)-\cC_p
\]
of some closed solid double cone $\cC_p$ with axis being the normal line
to $D(p)$ at $p$. In other words,
after replacing by a subsequence,
the surfaces ${M_{n}}$ have locally bounded norm of the second
fundamental form in $W(p)$, which implies $W(p)\cap\cS=\O$.
This observation implies that for every point
$p\in \cS$, one has a minimal lamination $\cL_{W(p)}=\cL\cap W(p)$
of $W(p)$ as  described in previous paragraphs: see
Figure~\ref{fig2cone}. This local picture is exactly the same as the one that
occurs in
the case where the mean curvatures
of the surfaces ${M_{n}}$ are zero; as in the
minimal case, we refer to it as the
local {\em Colding-Minicozzi picture} near points in $\cS$.
  See the discussion following Definition~4.9  in~\cite{mpr14}
for a more detailed analysis  of this picture.

\begin{definition}
Given $p\in \cS$, let $L_p$ be the leaf of $\cL$ containing the
punctured disk $D(p,*)$.
\end{definition}

The
arguments appearing in the proof of this claim are based on the proof of the similar
Lemmas~4.10, 4.11 and 4.12 in~\cite{mpr14}.

\begin{claim}\label{epsdistance}
The closure of  $L_p$ in $\rth$ is
a  plane $\ov L_p$ which intersects $\cS$
in a discrete set of points.
\end{claim}

\begin{proof}
We claim that  the punctured disk $D(p,*)$ in the Colding-Minicozzi
picture at $p$ is  a limit leaf of the local
lamination $\cL_W(p)$ of $W(p)$. Recall that in
$W(p)$ the surfaces $M_n$ have uniformly bounded norm of the second fundamental
form at points of intersection with the
annulus $A=W(p)\cap \partial \B(p,\ve)$. After replacing $\cC_p$ by a cone of wider aperture and choosing
$\ve>0$ sufficiently small, for $n$ sufficiently large, the annulus
$A$ contains a pair of
spiraling arcs $\a_1(n), \a_2(n)\subset A\cap M_n$
that begin
at one of the  boundary components of $A$ and end at its other boundary component.
Furthermore, as $n\to \infty$, the arcs  $\a_1(n), \a_2(n)$ converge to
a limit lamination $\cL_A$ of $A$ that contains
the simple closed curve $D(p,*)\cap A$, which  is a graph of small gradient
over its projection to the tangent space of $D(p)$ at $p$.  Since every
homotopically non-trivial simple closed curve in $A$ intersects $\a_1(n)$,
then by compactness, such a simple closed curve also intersects  $\cL_A$.  In particular,
$D(p,*)\cap A$ must be a limit leaf of  $\cL_A$.
It follows that $D(p,*)$ is a limit leaf of $\cL_W(p)$ of $W(p)$,
which proves our claim.

Since the punctured disk $D(p,*)$ in the Colding-Minicozzi
picture at $p$ is  a limit leaf of the local
lamination $\cL_W(p)$ of $W(p)$, $L_p$ is a limit leaf of
$\cL$, and thus it is stable. Consider $L_p$ to be a Riemannian surface
with its related metric space structure, namely,  the
distance between two points in $L_p$
is the infimum of the lengths of arcs on the leaf that
join the two points. Let $\wh{L}_p$
be the abstract metric completion of ${L_p}$. Since $L_p$ is a subset of
$\rth$, $\rth$ is complete and extrinsic distances are at most equal to
intrinsic distances, then the inclusion map of $L_p$ into $\rth$ extends uniquely
to a continuous map from  $\wh{L}_p$ into $\rth$,
and the image of $\wh{L}_p$ is contained in the
closure $\ov L_p$ of $L_p$ in $\rth$. Note that this continuous map
sends a point $q\in  \wh{L}_p -L_p$ to a point
of $\ov L_p\cap \cS$, which with an abuse
of notation we still call $q$. Suppose
$q\in \ov L_p\cap \cS$ is the induced inclusion
into $\rth$  of a point in  $\wh{L}_p$
and let
$\{q_k\}_k\in {L_p}$ be a Cauchy sequence converging to $q$.
If for all $q\in \ov L_p\cap \cS$, the related Cauchy
sequence $q_k$ lies in the punctured disk
$D(q,*)$, then   the inclusion of the completion
$\wh{L}_p$ of $L_p$ in $\rth$ would be a complete
minimal surface  in $\rth$.  Since  $\wh{L}_p$
would be stable outside of a discrete set of points and since for any compact minimal
surface $\Lambda$ with boundary, $\Lambda $ is stable if
and only $\Lambda $  punctured in a finite set of points is stable, then it
follows that the minimal surface  $\wh{L}_p$  is stable.
Hence,  $\wh{L}_p$ viewed in $\rth$ would be a plane equal to $\ov L_p$~\cite{cp1,fs1}.
Thus, in order to show that $\ov L_p$ is a plane, it suffices to show
that for $k$ large, the points $q_k$ lie in the punctured disk
$D(q,*)$.

Arguing by contradiction, suppose that for some $q\in \ov L_p\cap \cS$
and $k$ large, $q_k\not\in D(q,*)$.
Clearly for $k$ large, $q_k$ is arbitrarily close to $q$
in $\rth$ and in particular, $q_k\in \B(q,s)$. Then it
follows that, after extracting a subsequence, for $k$ large,
the points $q_k$ lie in the same component $\Delta$ of
$\B(q,s)-D(q)$.
First consider the
special case where $q$ is an isolated point in $\cS\cap \ov{\Delta}$,
that is, for $\rho$ sufficiently small,
$ \ov{\Delta}\cap \cS\cap \B(q,\rho)=\{q\}$. Then
$\ov L_p\cap  \ov{\Delta}\cap[\B(q,\rho)-q]$
is a minimal lamination of  $\B(q,\rho) -\{q\}$ with stable leaves and thus
$q$ is a removable singularity of this minimal
lamination by Theorem~1.2 in~\cite{mpr10}.
This regularity property implies that for $\rho$ sufficiently small,
$\overline \Delta\cap \{L_p\cup \{q\}\}\cap\B(q,\rho)$
contains a collection of disks $\{D_n\}_{n\in \N}$ in $\cL$
with boundary curves in $\partial\B(q,\rho)$
that converge
$C^1$ to the disk $D(q)$ as $n$ goes to infinity.
Since the points $q_k$ lie in  components of
$\overline \Delta \cap \{L_p\cup \{q\}\}$ that are different from $D(q)$,
then their intrinsic distances to $q$ in $\widehat L_p$ would be
bounded uniformly from below by $\rho/2$ for $k$ large; this is  because each such point $q_k$
is separated in $\B(q,\rho)$ from $q$ by the disk ${D}_n$ for $n$ sufficiently large.
Therefore, in this case the sequence of points $q_k$ cannot be a Cauchy
sequence converging to $q$ in $\widehat L_p$.

The case when $q$ is not an isolated point of
$\cS\cap \ov{\Delta}$ can be treated in a
similar manner. If there exists a sequence of points
$p_i\in \cS\cap \ov{\Delta}$ converging to $q$, then for $i$ large, through
each of these points there would be a punctured disk
$[D'(p_i)-\{p_i\}]\subset L_q$  punctured at $p_i$
with boundary in $\partial \B(p,\rho)$
and the disks $D'(p_i)$ converge $C^1$ to $D(p)$.
If two points, $x_1, x_2$, in $L_p\cap  \Delta\cap \B(q,\rho/2)$, $\rho$ small, lie on
different disks or are separated in $\Delta$ by one of the disks $D'(p_i)$, then the
distance between $x_1$ and $x_2$ in $L_p$ would be bounded from below by
$\rho/3$ for $k$ large.  Since the points $p_i$ converge in $\Delta\cap \B(q,\rho/2)$
to $q$, there is always a disk $D'(p_i)$ that eventually
separates points in the Cauchy sequence  $\{q_k\}_k$. Hence,
the sequence $\{q_k\}_k$ cannot be a Cauchy sequence unless,
for $k$ large, the points $q_k$
lie in  $D(q,*)$. By the argument given in the first paragraph of
the proof, this implies that $\ov L_p$ is a plane.

Finally, the fact that $\ov L_p\cap \cS$
is a discrete set of points in the plane  $\ov L_p$ follows
 from the geometry of the Colding-Minicozzi picture.
This completes the proof of the claim.
\end{proof}

By Claim~\ref{epsdistance}, for any $p\in \cS$
the closure of  $L_p$ in $\rth$ is
a  plane $\ov L_p$ which intersects $\cS$
in a discrete, therefore countable, set of points. After
applying a fixed rotation around the origin,
we will assume that $L_{\vec{0}}$ and $L_p$, for
any $p\in\cS$, are horizontal planes.

The  arguments already considered
in the proof of Theorem~\ref{H-lam-thm}  can be adapted to show
that the leaves of the minimal lamination
$\cL$ have genus at most $k$
and hence have  finite genus. For the sake of
completeness, we include the proof
of this key topological property for the leaves of $\cL$.

\begin{claim} \label{FG-claim}
Each leaf $L$ of $\cL$ has finite genus at most $k$.
\end{claim}

\begin{proof}
First suppose that $L=L_p$ for some $p\in \cS$. In this case $\ov L_p$ is a
plane and so $L$ has genus zero, which implies that the claim holds for $L$.

Next consider the case $\ov{L}\cap \cS=\O$.  In this
case $\ov{L}$ is a minimal lamination
of $\rth$ and the arguments in the proof of
Theorem~\ref{H-lam-thm} imply that the claim holds for $L$.

Finally, consider the case that
$p\in \ov{L}\cap \cS$ and $L\neq L_p$.  In this case,
$L$ lies in a halfspace component  $\cH$ of $\rth-\ov L_p$.
Suppose for the moment
that $  (\ov{L}\cap \cS)-\ov L_p\neq \O$ and let
$q\in(\ov{L}\cap \cS)-\ov L_p$.  In this subcase,
$L$ is contained in the open slab $\cT$ of $\rth$
with boundary planes $\ov L_p$ and $\ov L_q$.
Since $L$ is connected, it must intersect every horizontal
plane contained in $\cT$ and $\cT\cap \cS =\O$.

We claim that $L$ is properly embedded in $\cT$.  If not,
then the closure of $L$ in $\cT$ is a minimal
lamination of $\cT$ with a limit leaf $X$, which is stable.
By the same argument as in Claim~\ref{epsdistance}, stability implies that $X$ extends
across the closed countable set
$\ov{L}\cap \cS\subset(\ov L_p\cup \ov L_q)\cap \cS$ to a
complete stable minimal surface
in $\rth$.  Hence, $X$ is a horizontal plane
in $\cT$ which is disjoint from $L$, which contradicts the
discussion in the previous paragraph.
Hence, in this subcase $L$ is properly embedded in
$\cT$.
A similar argument shows that if $  (\ov{L}\cap \cS)-\ov L_p=\O$,
then the leaf $L$ is properly embedded in the half space $\cH$.  Hence,
in either case, $L$ is properly embedded in an open simply-connected
subset of $\rth$ and so it separates this
open set and has an open regular neighborhood in it.  Now the
 arguments in the proof of Theorem~\ref{H-lam-thm} imply that $L$ has finite
 genus at most $k$.
\end{proof}

By using Claims~\ref{epsdistance} and~\ref{FG-claim}, the next claim follows. Since the proof of
Claim~\ref{onepoint} is almost identical to the proof
of the first statement in  Claim~3.2 in~\cite{mt13}, we omit it here.

\begin{claim} \label{onepoint}
For any $t\in\mathbb{R}$, the intersection $\{x_3=t\}\cap\cS$ is nonempty.
\end{claim}

We now   invoke the last hypothesis in the statement of Theorem~\ref{geometry1}:
$$\lim_{n\to \infty} I_{{M_{n}}}(\vec{0})=\infty.$$
We remark that we have obtained Claim~\ref{onepoint} without invoking this hypothesis.
After replacing by a subsequence, for each $m\in \N$, there exists  an
 increasing sequence $N(m)\in \N$ such that the
injectivity radius function of $M_{N(m)}$ at $\vec{0}$ is
greater than ${m}/\de_1$,
where $\de_1$ is the constant given in Theorem~\ref{thm1.1}.
Therefore, by the same theorem, for any $m\in\N$,
the connected component $M(N(m))$ of
$M_{N(m)}\cap \B(m)$ containing the origin is an $H_{N(m)}$-disk with
$\partial M(N(m))\subset \partial \B(m)$.  For simplicity of notation and
after replacing by a further subsequence and relabeling,
we will use $M_n$ to denote the  sequence
$M_{N(m)}$ and so $M(n)$ will now denote $M(N(m))$.
After replacing by a further sequence, we will assume that item~B
of Theorem~\ref{thm2.1} holds.

Let $l$ denote the $x_3$-axis. Since
$\lim_{n\to \infty}|A_{M_n}|(\vec{0}) = \infty$, part (a) of item~B
of Theorem~\ref{thm2.1} shows that
the sequence $M(n)$ converges
away from $l$ to a foliation $\cL'$ of $\rth-l$ by punctured
horizontal planes. Part (b) of
item~B of Theorem~\ref{thm2.1} implies that given $R>0$,
if $n$ is sufficiently large
there exists a possibly disconnected compact subdomain
$\cC_{n}(R)$ of $M(n)$, with $[M(n)\cap \B(R/2)]\subset \cC_n(R) \subset \B(R)$
and with $\partial \cC_n(R)\subset \B(R)-\B(R/2)$,
consisting of a disk $\cD_n(R,1)$ containing the origin $\vec{0}$ and
possibly a second disk  $\cD_n(R,2)$. Moreover, the diameter of each connected component of
$\cC_{n}(R)$ is bounded by $3R$ and $\cD_n(R,i)\cap \B(R/n)\neq\mbox{\rm \O}$, for $i=1,2$.
Hence, if
$M_n\cap \B(R/2)=M(n)\cap \B(R/2)$ then the theorem follows. If that is not the case,
then we proceed as follows.

Suppose, after choosing a subsequence, that  for some $R>0$,
$M_n\cap \B(R/2)$ contains a component $\Delta_n(R)$ that
is not contained in  $\cC_n(R)$.
We first  show that even in this case, the sequence
$\{M_n\}_{n\in \N}$,
and not solely $\{M(n)\}_{n\in \N}$, converges to
the foliation $\cL'$ away from $l$. Since $M_n$
has locally positive injectivity
radius, the horizontal planar regions forming
on $M(n)$ away from $l$ imply that
the sequence
$\{M_n\}_{n\in\mathbb N}$ has locally bounded norm of the
second fundamental form in $\rth-l$; these
curvature estimates arise from the
intrinsic curvature estimates
in Corollary~\ref{cest2}. By the embeddedness of $M_n$, $M_n$
must converge to  $\cL'$ away
from $l$ as $n$ goes to infinity
and $l$ is again a line and nearby it the sequence $M_n$
has arbitrary large norm of the second fundamental form.
This discussion  proves that $\cS=l$ and $\cL=\cL'$
regardless of whether or not $M_n\cap \B(R/2)=M(n)\cap \B(R/2)$.
Moreover, using these curvature estimates and the double
spiral staircase structure of $\cD_n(R,1)$,  it is
straightforward to prove  that
  $\Delta_n(R)$ contains points $y_n$ converging
to $\vec{0}$.

After choosing a subsequence, for some
$R$ fixed and for every $n\in \N$,
$M_n\cap \B(R/2)\neq M(n)\cap \B(R/2)$. In this remaining case let
 $y_n$ be chosen as in the previous paragraph.

Assume that $\lim_{n} I_{M_n}(y_n)=\infty$; in fact, we
will prove this in Claim~\ref{claim:Inj=Inf}. Arguing similarly
to the previous discussion, after replacing by a subsequence,
we may  assume that $I_{M_n}(y_n)\geq n/\de_1$, where $\de_1$
is the constant given in Theorem~\ref{thm1.1}, and that
$\partial M(n)\subset\partial \B(R_n)$, with $R_n>2n$.
By Theorem~\ref{thm1.1}
the connected component $M'(n)$ of
$M_n\cap \B(y_n,n)$ containing $y_n$ is an $H_n$-disk with
$\partial M'(n)\subset \partial \B(y_n,n)$.
Item~B of
Theorem~\ref{thm2.1} implies that for $n$ sufficiently large,
there exists a possibly disconnected compact subdomain
$\cC_n'(R)\subset M'(n)$ with $\cC_n'(R) \subset \B(y_n,R)$
and with $\partial \cC'_n(R)\subset [\B(y_n,R)-\B(y_n,R/2)]$ consisting
of a disk $\cD'_n(R,1)$ containing  $y_n$ and possibly
a second disk $\cD'_n(R,2)$, where
each disk has intrinsic diameter bounded by $3R$
and $\cD'_n(R,i)\cap \B(y_n,R/n)\neq\mbox{\rm \O}$, for $i=1,2$.

Since $\lim_{n\to\infty}y_n=\vec 0$ and $R_n>2n$, $M(n)$
and $M'(n)$ are disks satisfying the following properties:
\bit
\item $M(n)\subset \B(R_n)$ and $\partial M(n)\subset \partial \B(R_n)$;
\item $M'(n)\subset \B(y_n, n)\subset \B(R_n)$ and
$\partial M'(n)\subset \partial \B(y_n,n)$;
\item $y_n\notin M(n)$ and $y_n\in M'(n)$.
\eit
 Then elementary separation properties give
 that $M(n)\cap M'(n)=\mbox{\rm \O}$. In particular,
 $\cC_n(R)\cap \cC'_n(R)=\mbox{\rm \O}$. Thus, to finish the
 proof assuming $\lim_{n\to \infty} I_{M_n}(y_n)=\infty$, it suffices to show that
$M_n\cap \B(R/2)\subset \cD_n(R,1)\cup\cD_n'(R,1)$.

Suppose that either $\cD_n(R,2)$ or $\cD_n'(R,2)$ existed.
Applying Corollary~\ref{cest-cor} would give that $\vec 0$
cannot be a singular point. Therefore,  $\cD_n(R,2)$ and
$\cD_n'(R,2)$ do not exist.
On the other hand, if  $M_n\cap \B(R/2)\neq[M(n)\cup M'(n)]\cap \B(R/2)$,
then by repeating the arguments used so far, there would exist
a sequence of point $x_n$ with $\lim_{n\to\infty}x_n=\vec 0$
and a third sequence of disks $\cD''_n(R)$ disjoint from
$\cD_n(R,1)\cup\cD_n'(R,1)$, with $x_n\in \cD''_n(R)$ and
$\partial \cD''_n(R)\subset [\B(x_n, R)-\B(x_n,R/2)]$.
Again, one would obtain a contradiction by applying Corollary~\ref{cest-cor}.
Therefore $M_n\cap \B(R/2)=[M(n)\cup M'(n)]\cap \B(R/2)$ and so, to
complete the proof of Theorem~\ref{geometry1}, it remains to prove the claim below.

\begin{claim} \label{claim:Inj=Inf}
$\lim_{n\to \infty} I_{M_n}(y_n)=\infty$.
\end{claim}
\begin{proof}
Arguing by contradiction, suppose that
after replacing by a subsequence
$\lim_{n\to \infty} [I_{M_n}(y_n)=T_n]=T\in (0,\infty)$.
Since $\lim_{n\to\infty} H_n=0$, then,
 for $n$ large, the Gauss equation implies that
 the $\limsup K_{M_n}$  of the Gaussian curvature
 functions of the surfaces $M_n$ is non-positive.
 Classical results on Jacobi fields along geodesics
in such surfaces imply that for $n$ large the exponential map of $M_n$ at $y_n$
 on the closed disk in $T_{y_n}M_n$ of radius $T_n$
is a local diffeomorphism that is injective on the interior of
the disk but it is not injective along its boundary
circle of radius $T_n$. Hence,
there exists a sequence of simple closed geodesic loops
$\a_n\subset M_n$ based at $y_n$ and
of lengths $2T_n$ converging to $2T$ that are smooth everywhere
except possibly at $y_n$. Since the sequence  $\{M_n\}_{n\in\mathbb N}$
has locally positive injectivity radius in $\rth$,
there exists an $\ve\in(0,T)$ such that for $n$ large,
\[
I_{M_n}|_{M_n\cap\B(\vec 0, 5T)}\geq \ve.
\]
Therefore, if the intrinsic distance between two points $x$
and $y$ in $M_n\cap \B(\vec 0, 5T)$ is less
than $\ve$, then there exists a
unique length minimizing  geodesic in $B_{M_n}(x,\ve)$
 connecting them.

Since
 the sequence of
surfaces is converging   to  flat planes away
from $l$ and $\lim_{n\to\infty }y_n= \vec 0$,
if for some divergent sequence of integers $n$,
there were points
$p_n\in \a_n$ that lie outside of some fixed
sized cylindrical neighborhood of $l$ and converge to a point $p$,
then a subsequence of the geodesics $\a_n$ would
converge to a set containing
an infinite geodesic
ray starting at $p$ in the horizontal plane containing $p$. This follows
because the converge is smooth away from $l$. If there were a sequence of
points $p_n\in\a_n$ converging to a point $p$ not in $l$ then a neighborhood
$U_n$ of $p_n$ would converge smoothly to a horizontal flat disk $D(p)$ centered
at $p$. Since $\a_n$ is a geodesic, $\a_n\cap U_n$ would converge to a diameter $d$
of $D(p)$ and there would be a point $q\in D(p)$ which is the limit of
points $q_n\in\a_n$ and that is further away from $l$ then $p$. The convergence
is smooth nearby $q$. Therefore, applying the previous argument gives that the limit
set of convergence of $\a_n$ can be extended at $q$ in the direction $\overrightarrow{pq}$.
Iterating this argument would give that the limit set of $\a_n$ contains an infinite geodesic
ray starting at $p$ in the horizontal plane containing $p$. This would give a
contradiction because the loops have
length less than $3T$. Therefore after replacing by a
subsequence, the $\a_n$ must converge to a vertical
segment $\sigma$ containing the origin and of length
less than or equal to $T$.

Note that by Theorem~\ref{thm1.1}, for $n$
large, $\a_n$ cannot be contained in
$\B(y_n, \de_1 \ve)$, otherwise it would be contained in
$B_{M_n}(y_n, \ve \slash 2)$ and, by the properties of $\ve$,
$B_{M_n}(y_n, \ve \slash 2)$ cannot contain
a geodesic loop such as $\a_n$. Therefore, if we
let $p_1$ and $p_2$
be the endpoints of the line segment $\sigma$,
without loss of generality, we can assume
that $p_1\neq \vec 0$ and that $x_3(p_1) \in [\delta_1 \ve,T]$.
Let $q_n$ be  points of $\a_n$ with a
largest $x_3 $-coordinate, and so $\lim_{n\to\infty}q_n=p_1$.
By arguments similar to the ones used in the previous paragraphs of
this proof, for any $r< \ve/ 2$,
$\a_n\cap \B(q_n, \de_1r)$ contains an arc component $\be_n$ with
$\partial \be_n= z_n(1)\cup z_n(2)\subset \partial \B(q_n, \de_1r) $
satisfying the following
properties for $n$ large:
\begin{enumerate}
\item $q_n\in\be_n$;
\item $\lim_{n\to\infty}|z_n(1)-z_n(2)|=0$;
\item $z_n(2)\in B_{M_n}(z_n(1), r) $
\item $\mbox{\rm dist}_{M_n} (z_n(1),z_n(2))\geq \de_1 r.$
\end{enumerate}

  Let $r:=\min \{\ve\slash a, \ve/ 2\}$, where $a$ is the
  constant given in Theorem~\ref{main2}.
  Then applying Theorem~\ref{main2} with $\vec 0$ replaced
  by $z_n(1)$ and $R=r$, we have that
if $\sup_{B_{\S}(z_n(1),r_0(n))}|A_{M_n}|>\frac1{r_0(n)}$
where $r>r_0(n)$,
then
\[
\frac{1}{3}\mbox{\rm dist}_{M_n} (z_n(1),z_n(2))<|z_n(1)-z_n(2)|+r_0(n).
\]
Since as $n$ goes to infinity, there are points
arbitrarily intrinsically close to
$z_n(1)$ and with arbitrarily large norm of the second fundamental form,
we can assume that $r_0(n)<\frac{ \de_1 \ve}{6 a} $. Combining  this,
$\mbox{\rm dist}_{M_n} (z_n(1),z_n(2))\geq \frac{ \de_1 \ve}{a}$ and the
previous inequality,   we have obtained that
\[
\frac{ \de_1 \ve}{6 a}<|z_n(1)-z_n(2)|.
\]
Since the right hand-side of this inequality is
going to zero as $n$ goes to infinity,
while the left hand-side is fixed, bounded away
from zero, independently on $n$, we
have obtained a contradiction, which  finishes the proof that
$\lim_{n \to \infty} I_{M_n}(y_n)=\infty$.
\end{proof}

Now that item~2 of Theorem~\ref{geometry1} is proved, we can apply Theorem~\ref{thm2.1} and
Remark~\ref{remark:spiral}  to obtain the double spiral staircase description in item~3 for
the each of the 1 or 2  components of $\cC_n$.

This final observation completes the proof of Theorem~\ref{geometry1}.

\section{The proof of Theorem~\ref{geometry2}.} \label{sec5}

In this section we will prove Theorem~\ref{geometry2}.
Suppose that  $\{M_n\}_{n\in \mathbb N}$
is a sequence of compact $H_n$-surfaces in $\rth$
with finite genus at most $k$, $\vec{0}\in {M_{n}}$, ${M_{n}}$
contains no spherical components,
$\partial {M_{n}}\subset [\rth -\B(n)]$, the sequence has locally
positive injectivity radius in $\rth$ and
\[
\lim_{n\to \infty} |A_{{M_{n}}}|(\vec{0})=\infty\,\text{ and
}\lim_{n\to \infty} I_{{M_{n}}}(\vec{0})=C,\]
 for some $C>0$.

After replacing by a subsequence,  Claim~\ref{onepoint} implies that
the surfaces $M_n$ converge $C^\alpha$, for any $\a\in(0,1)$, to a minimal
lamination $\cL$ outside of a closed
set $\cS$ with $\vec 0\in \cS$ and $x_3(\cS)=\R$, where each leaf of
$\cL$ is a horizontal plane punctured in a discrete set of points in $\cS$.

The Colding-Minicozzi picture of $\cL$ around
each point of $\cS$ together with
the curvature estimates in Theorem~\ref{th} and the fact that
the foliation $\cL$ is a foliation of $\rth-\cS$ by
punctured horizontal planes imply that if
$\cS_0$ is a connected component of $\cS$, then $x_3(\cS_0)=\R$ and $\cS_0$
is a Lipschitz graph over the $x_3$-axis.  Moreover,
given $p\in \cS \cap \B(R)$, there exists $\delta:=\delta(R)>0$ such
that for $n$ large, the intersection
${M_{n}}\cap \B(p,\delta)$ consists of one or two disks. This is a
consequence of Corollary~\ref{cest-cor}, Theorem~\ref{thm1.1} and the
fact that for $n$ large, the injectivity radius
function of ${M_{n}}$ is bounded away from zero on
any fixed compact
set of $\rth$. By this observation and arguing
like in the proof of Theorem~1.1 in~\cite{mt13},
one obtains the following  result.

\begin{claim} \label{claim-lines} The set $\cS$
satisfies the following properties:
\bit
\item The set $\cS$ is a discrete collection of vertical
lines, one of which is the $x_3$-axis.
\item Given $R>0$, for $n$ sufficiently large
the intersection of each line segment  in $\cS\cap \B(R)$
is the $C^1$ limit with multiplicity at most two of analytic curves in $M_n$
which  are pre-images of the equator via the Gauss map.
\item Let $l$ be a line in $\cS$. Given $p\in l$
and $R>0$ such that $\B(p,R)\cap \cS \subset l$
then, for $n$ large, the  collection  $\cC_n$ of components of
${M_{n}}\cap \B(p,\frac R2)$ such that
$\cC_n\cap \B(p,\frac R4)\neq \mbox{\rm \O}$ consists of at most two disjoint disks.
Furthermore, each of the 1 or 2 disk components of \,$\cC_n$ is
contained in a   disk in $M_n\cap \B(p,R)$ with boundary curve in $\B(p,R)-\B(p,R/2)$, where these   disks 
have  the structure of
double spiral staircases, see Remark~\ref{remark:spiral}, with central columns that are graphs with small $C^1$-norms over an arc in $l\,\cap \B(R)$,
and $\cC_n\cap \B(p,\frac R4)$ is contained in the union of these subdisks.
\eit
\end{claim}

Let $l$ be a line in $\cS$. We now need to attach two labels to $l$.
The first one is the following: if $l$
is the $C^1$ limit with multiplicity one, respectively two, of analytic curves
which  are pre-images of the equator via the Gauss map, we say
that $l$ has {\em  multiplicity one}, respectively {\em two}.
The second label is the following: let $C_l(R)$ be the
vertical solid cylinder of radius $R$ with axis $l$.
For a given line $l$ in $ \cS$, fix $R_l>0$ such that
$C_l(2R_l)\cap \cS=l$. Then, for $n$ large,
$\partial C_l(R_l\slash2)\cap {M_{n}}$ contains  either
two or four highly winding spirals nearby the $(x_1,x_2)$-plane.
Since ${M_{n}}$ is embedded,
these spirals are all right-handed
or left-handed for a given $n$. After passing to a subsequence
and using a diagonal argument gives that for a given  $l$ in $\cS$ and $n$ large,
the spirals have the same ``handedness.'' We say that
$l$ is {\em right-handed} if such spirals
are right-handed and that $l$ is {\em left-handed} otherwise.

In the next claims we prove that $\cS$ consists of exactly two vertical
lines.

\begin{claim}\label{twolines}
Let $l_1$ and $l_2$ be two distinct components of $\cS$. Then, if $\, \! $ $l_1$
is right-handed (left-handed), $l_2$ must be left-handed (right-handed).
In particular, $\cS$ consists of at most two lines.
\end{claim}
\begin{proof}
Arguing by contradiction, suppose that $l_1$ and $l_2$ are two
distinct components of $\cS$
having the same handedness.
We will obtain a  contradiction by proving that as
$n$ goes to infinity, the  number of pairwise disjoint pairs of
loops in $M_n$ such that each pair intersects transversely at one point is
greater than the fixed genus bound $k$ for the surfaces ${M_{n}}$.
By using standard topological arguments, the existence of such
loops implies that the genus ${M_{n}}$ is greater than  $k$.

Without loss of generality, suppose that $l_1$ and $l_2$ are both left-handed. Let
$p_i:=l_i\cap\{x_3=0\}$, $i=1,2$  and let $\overline{p_1p_2}$
denote the line segment connecting them. For simplicity, first assume that
$\overline{p_1p_2}\cap [\cS-[l_1\cup l_2]]=\O$.
Then,  as $n$ goes to infinity, the segment
$\overline{p_1p_2}-[C_{l_1}(R_{l_1}\slash4 )\cup C_{l_2}(R_{l_2}\slash 4)]$
lifts near the $(x_1,x_2)$-plane to an increasing number of
arcs $\gamma_i$ in
${M_{n}}-[C_{l_1}(R_{l_1}\slash4 )\cup C_{l_2}(R_{l_2}\slash 4)]$. In fact,
an $\ve$-neighborhood  $\Gamma$ of $\overline{p_1p_2}$ in the $(x_1,x_2)$-plane
lifts to an increasing number of strips $\Gamma_i$ in
${M_{n}}-[C_{l_1}(R_{l_1}\slash4 -\ve)\cup C_{l_2}(R_{l_2}\slash 4-\ve)]$.
Because ${M_{n}}$ is embedded,
the strips $\Gamma_i$ can be ordered by their relative heights. Moreover,
the arcs of the spiralling curves
$M_n \cap \partial C_{l_1}(R_{l_1}\slash2 )$ given by $\Gamma_i\cap \partial C_{l_1}(R_{l_1}\slash2 )$
can be connected, via arcs in $\Gamma_i$,  to the arcs in the spiralling
curves  $M_n \cap \partial  C_{l_2}(R_{l_2}\slash2 ))$ given by $\Gamma_i\cap \partial C_{l_1}(R_{l_2}\slash2 )$.
There are three possibilities to consider.
\begin{enumerate}
\item The lines $l_1$ and $l_2$ have both multiplicity one.
\item One line has multiplicity one and the other one has multiplicity two.
\item Both lines have multiplicity two.
\end{enumerate}

The construction  of the collection of pairwise disjoint  pairs
 of loops when the lines  $l_1$ and $l_2$ have both
 multiplicity one is illustrated
in Figure~\ref{infinitegenus1}.
\begin{figure}
\begin{center}
\includegraphics[width=12cm]{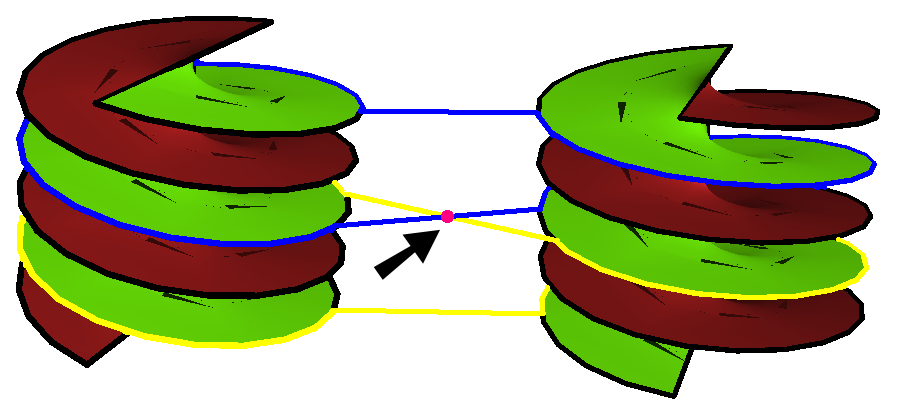}
\caption{The blue curve and the yellow curve intersect exactly at one point.}
\label{infinitegenus1}
\end{center}
\end{figure}
 The construction of the collection of pairwise disjoint  pairs
 of loops in case two is illustrated in Figure~\ref{infinitegenus2}.
 \begin{figure}
\begin{center}
\includegraphics[width=12cm]{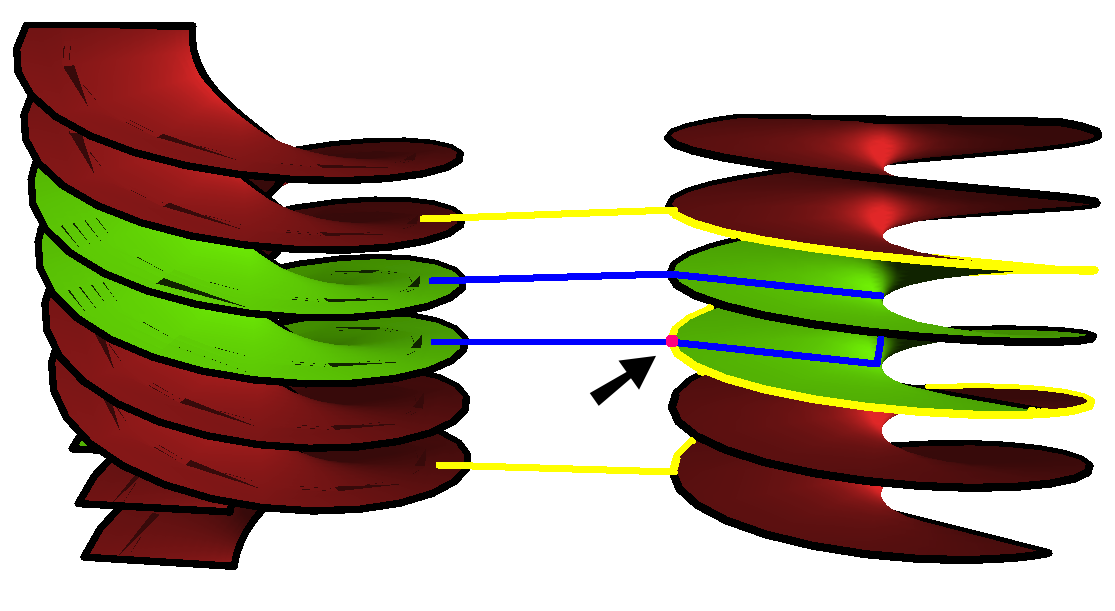}
\caption{The right side of the picture is connected. On the left side of the
picture, the red set is part of a connected set $\cH_1$ and the green set is
part of a connected set $\cH_2$. The sets $\cH_1$ and $\cH_2$ are disjoint.
Therefore, the end points of the blue arc and of the yellow arc can be connected
so that the resulting closed curves intersect
in exactly one point, as shown in the picture.}
\label{infinitegenus2}
\end{center}
\end{figure}
The construction in the third and last case is also
straightforward and it is left to the
reader.

If $\overline{p_1p_2}\cap [\cS-[l_1\cup l_2]]\neq\O$,
the proof can be easily modified by  replacing
$\overline{p_1p_2}$ by
a smooth embedded arc in the $(x_1,x_2)$-plane that is a small
normal graph over $\overline{p_1p_2}$ and only intersects
the singular set at its end points $p_1,p_2$.
\end{proof}

The next claim finally shows that $\cS$ consists of exactly two lines.
Note that by the previous claim, if there are two lines in $\cS$,
then one of these two lines
must be right-handed and the other one must be left-handed.
Recall that
\[
\lim_{n\to \infty} I_{{M_{n}}}(\vec{0})=C.
\]

\begin{claim}\label{twovertical}
The set $\cS$ consists of exactly two vertical lines one of
which is the $x_3$-axis.
\end{claim}

\begin{proof}

We have already shown that there are at most
two vertical lines in $\cS$ and that
 the $x_3$-axis is in $\cS$.
Since $\lim_{n\to\infty} H_n=0$, then,
 for $n$ large, the Gauss equation implies that
 the $\limsup K_{M_n}$  of the Gaussian curvature functions
 of the surfaces $M_n$ is non-positive.
 Since  $\lim_{n\to \infty} [I_{{M_{n}}}(\vec{0})=C_n]=C$,
 classical results on Jacobi fields along geodesics
in such surfaces imply that for $n$ large the exponential map of $M_n$ at the origin
 on the  closed disk in $T_{\vec 0}M_n$ of radius $C_n$
is a local diffeomorphism that is injective on the
open disk but not injective along its boundary
circle of radius $C_n$. Hence,
there exists a sequence of simple closed geodesic loops
$\a_n\subset M_n$ based at the origin and
of lengths $2C_n$ converging to $2C$ that are smooth everywhere
except possibly at the origin.
 Arguing by contradiction, suppose that $\cS$ is the $x_3$-axis.
The arguments to rule out this picture are exactly the same ones used in the proof
of Claim~\ref{claim:Inj=Inf} by taking $T=C$.
This implies that the number of lines in $\cS$ must be two.
\end{proof}

From now on, $l_1$ will denote the $x_3$-axis, $l_2$ will denote the
other component in $\cS$ and $p_2=l_2\cap \{x_3=0\}$.

\begin{claim}\label{multi}
If $l_1$ has multiplicity one, respectively two, then so does $l_2$.
\end{claim}
\begin{proof}
By Claim~\ref{twolines}, one vertical line must be right-handed and the
other must be left-handed. Suppose that one line has multiplicity one and the other
has multiplicity two. Like in the proof of Claim~\ref{twolines},
we will obtain a  contradiction by proving that as
$n$ goes to infinity, the  number of pairwise disjoint pairs of
loops such that each pair intersects transversely at one point is
greater than the fixed genus bound $k$ for the surfaces ${M_{n}}$.
The construction of such pairs of loops  is illustrated
in Figure~\ref{infinitegenus3}.
\begin{figure}
\begin{center}
\includegraphics[width=12cm]{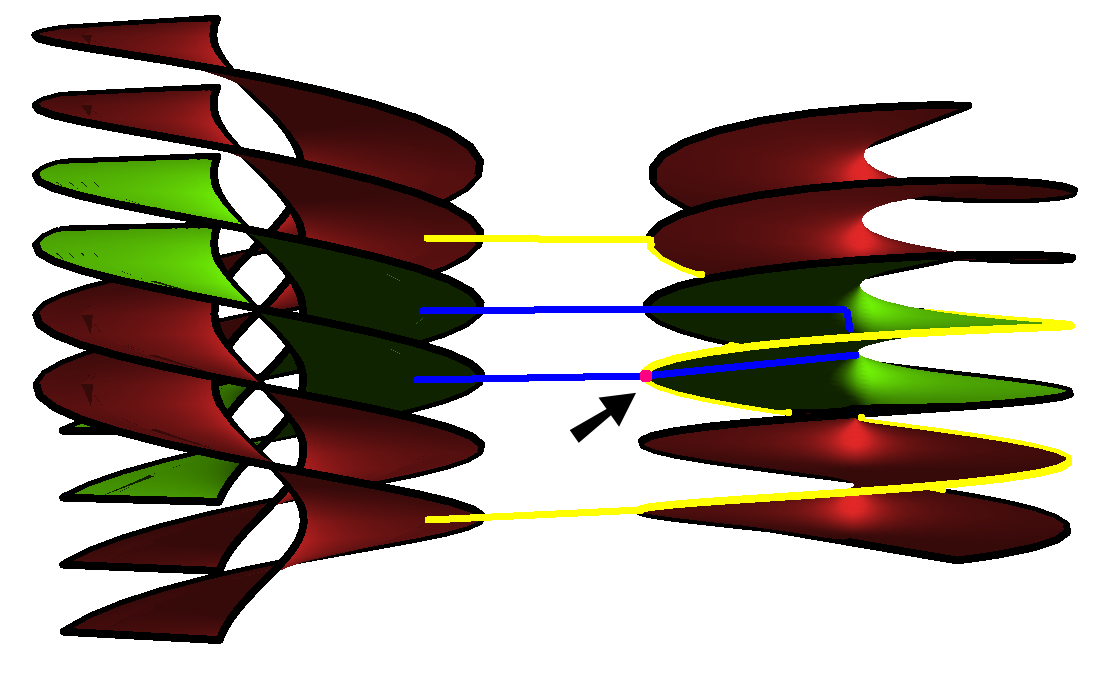}
\caption{The right side of the picture is connected. On the left side of the
picture, the red set is part of a connected set $\cH_1$ and the green set is
part of a connected set $\cH_2$. The sets $\cH_1$ and $\cH_2$ are disjoint.
Therefore, the end points of the blue arc and of the yellow arc can be connected
so that the resulting closed curves intersect in exactly one point, as shown in the
picture. One of the differences with Figure~\ref{infinitegenus2} is in the handedness of the lines.}
\label{infinitegenus3}
\end{center}
\end{figure}
\end{proof}

Assume now that $l_1$ has multiplicity two. Let $d>0$
denote the distance between $l_1$ and $l_2$.
Recall that by Claims~\ref{claim-lines} and~\ref{multi}, if
$l_1$ has multiplicity two so does $l_2$ and that, for $n$ large,
the subset $\cD_1(n)$ of $M_n\cap \B(\frac d2)$ such that
$\cD_1(n)\cap \B(\frac d4)\neq \O$ consists of two disks $A_1(n)$
and $B_1(n)$ and the subset $\cD_2(n)$ of $M_n\cap \B(p_2,\frac d2)$
such that $\cD_2(n)\cap \B(p_2,\frac d4)\neq \O$ consists of
two disks $A_2(n)$ and $B_2(n)$.

\begin{claim} \label{claim5.5}
Let $R>d$ and suppose that $l_1$ has multiplicity two. Then,
for $n$ large, after possibly relabeling the disks, the
following holds. The components  $\Delta_n$ of $M_n\cap \B(R)$
such that $\Delta_n \cap \B(\frac d4)\neq \O $ consist of two
distinct planar domains   $\Delta_1(n)$ and $\Delta_2(n)$
such that $A_1(n)\cup A_2(n)\subset \Delta_1(n)$ and
$B_1(n)\cup B_2(n)\subset \Delta_2(n)$.
 \end{claim}

 \begin{proof}
Let $\Pi$ denote the vertical plane perpendicular to the
segment $\vec 0 p_2$ connecting $\vec 0$ and $p_2$ and
containing its midpoint.  As $n$ goes to infinity,
$\Pi\cap \Delta_n$ consists of an increasing collection of
arcs $S(n)$ that are becoming horizontal arcs. Since $\Delta_n$
is embedded, these planar arcs can be ordered by
their relative heights over the midpoint of $\vec 0 p_2$.
Note that $\Delta_n-\Pi$ contains four disconnected components
$\Omega_{ij}(n)$, $i,j=1,2$, and the following holds for
$i=1,2$: $A_i(n)\subset \Omega_{i1}(n)$ and
$B_i(n)\subset \Omega_{i2}(n)$. For $i=1,2$ let $\alpha_i(n)$
denote the arcs in $S(n)$ that are contained in the boundary
of $\Omega_{i1}(n)$ and  let $\beta_i(n)$ denote the arcs in
$S(n)$ that are contained in the boundary of $\Omega_{i2}(n)$.
Recall that $A_1(n)$ and $B_1(n)$, respectively  $A_2(n)$ and $B_2(n)$,
separate $\B(\frac d4)$, respectively $\B(p_2,\frac d4)$,
into three components and the mean curvature vector of $M_n$
points outside of the component $W_1(n)$, respectively $W_2(n)$,
with boundary $A_1(n)\cup B_1(n)$, respectively $A_2(n)\cup B_2(n)$.
This is because otherwise applying Corollary~4.9 in~\cite{mt13}
would give curvature estimates in a neighborhood of $\vec 0$,
respectively $p_2$, and this would contradict the fact that
$\vec 0$, respectively $p_2$, is in $\cS$. Using this
observation and the previous discussion gives that
either $\alpha_1(n)=\alpha_2(n)$ and $\beta_1(n)=\beta_2(n)$,
or $\alpha_1(n)=\beta_2(n)$ and $\beta_1(n)=\alpha_2(n)$.
After possibly relabeling, either case implies that $\Delta_n$ is disconnected.

It remains to prove that each connected component of $\Delta_n$
has genus zero. This follows from the ``almost periodicity'' of
the previous description. If there were a pair of loops
intersecting at exactly one point then, as $n$ goes to infinity,
there would be an increasing number of such pairs in $M_n$,
contradicting the fact that the genus of $M_n$ is bounded from above by $k$.
 \end{proof}

\begin{remark}\label{moneortwo} {\em
If instead $l_1$ has multiplicity one then, by the same arguments,
the components  $\Delta_n$ of $M_n\cap \B(R)$
such that $\Delta_n \cap \B(\frac d4)\neq \O $ consist of a unique
 planar domain for $n$ large. Moreover, it is easy to see that if $l_1$ has
 multiplicity two and $\Delta_2(n)$ denotes the connected component of $\Delta_n$
 that intersect $\B(\frac d4)$ and does not contain the origin then, after possibly
 reindexing the subsequence, $\Delta_2(n)\cap \B(\frac Rn)\neq \O$.
}
 \end{remark}

For the time being, let us assume that the distance between $l_1$ and $l_2$ is $C$.
We now deal with the construction of closed curves
with non-zero flux. Note that this
construction  is analogous to the one
described in Figures~4, 5 and 6 in~\cite{mpr3}, which in turn was a
modification of a related argument in~\cite{mpr1}; the closed curves constructed
by the methods in ~\cite{mpr1,mpr3} are called {\em connection loops}.

Recall that if $\gamma$ is a 1-cycle in an $H$-surface $M$, then the %
{\em flux} of
$\gamma$ is
\[
F(\g)=\int_{\gamma}(H\gamma+\xi)\times \dot{\gamma},
\]
 where $\xi$
is the unit normal to $M$ along $\gamma$. The flux of a 1-cycle in $M$ is a homological invariant.

 Given $\ve>0$ sufficiently small,
as $n$ goes to infinity, the line segment
$\vec0p_2-[\B(\ve)\cup \B(p_2,\ve)]$ lifts to an
increasing number of arcs
$\gamma_i(n,\ve)$ in ${M_{n}}-[\B(\ve)\cup \B(p_2,\ve)]$
that, as $n$ goes to infinity,
converge $C^1$ to the line segment $\vec0 p_2-[\vec0 p_2\cap [\B(\ve)\cup\B(p_2,\ve)]]$.
Because ${M_{n}}$ is embedded, the lifts $\gamma_i(n,\ve)$
can be ordered by their relative heights and the signs of the inner
product between the unit normal vector to ${M_{n}}$ along
$\gamma_i(n,\ve)$ and $(0,0,1)$ are alternating.

 Let $\ve_n$ be a sequence of positive numbers with $\lim_{n\to\infty}\ve_n=0$
 such there exists a sequence of
two consecutive lifts $\gamma_1(n,\ve_n)$ and $\gamma_2(n,\ve_n)$ of
$\vec0p_2-[\B(\ve_n)\cup \B(p_2,\ve_n)]$ and the following holds: the end points
of such lifts are contained in $\B(2\ve_n)$ and $\B(p_2,2\ve_n)$ and the lifts
converge to the line segment $\vec 0p_2$ away from $\vec 0$ and $p_2$ as $n$ goes to infinity.
Let $\alpha_1(n,\ve_n)$
be an arc in $\B(2\ve_n)\cap {M_{n}}$
connecting the endpoints of $\gamma_1(n,\ve_n)$ and $\gamma_2(n,\ve_n)$
in $\B(\ve_n)$ and let $\alpha_2(n,\ve_n)$
be an arc in $\B(p_2,2\ve_n)\cap {M_{n}}$
connecting the endpoints of $\gamma_1(n,\ve_n)$ and $\gamma_2(n,\ve_n)$ in
$\B(p_2,2\ve_n)$ such that the loop
\[
\gamma_1(n,\ve_n)\cup\alpha_1(n,\ve_n)\cup \gamma_2(n,\ve_n)\cup\alpha_2(n,\ve_n)
\]
is smooth; note that  since the sequence  $\{M_n\}_{n\in\mathbb N}$ has locally
positive injectivity radius in $\rth$, by using Theorem~\ref{thm1.1}, as $n$ goes
to infinity, the sum of the lengths of the arcs $\alpha_1(n,\ve_n)$ and $\alpha_2(n,\ve_n)$,
can be assumed to approach zero as well.
Let $\Gamma_{n}$ be a unit speed parametrization
of such a loop and, for $p\in \Gamma_{n}$ let
$ N_{n}(p)$ denote the normal to ${M_{n}}$ at $p$.

Recall that as $n$ goes to
infinity, the mean curvature of $M_n$ is going to zero therefore, since the length
of $\Gamma_{n}$  is bounded from above independently of $n$,
the term in the flux formula  involving the mean curvature is going to zero.
In other words,
\[
F(\Gamma_{n})=\int_{\Gamma_{(n,\ve)}} N_{n}(p)\times
\dot \Gamma_{n)}(p)+f(n), \quad {\rm where } \lim_{n\to \infty}f(n)=0.
\]
As $n$ goes to infinity,
for any $p\in \gamma_i(n,\ve_n)$ the vectors
$N_{n}(p)\times \dot \Gamma_{n}(p)$ are
converging to the same unit vector perpendicular to $\vec0p_2$, the
lengths of $\a_1(n,\ve_n)\cup\a_2(n,\ve_n)$
are going to zero, and the lengths of
$\gamma_1(n,\ve_n)\cup\gamma_2(n,\ve_n)$ are converging to $2C$.
Therefore, after possibly changing their orientation, the
curves $\Gamma_{n}$ converge to
the line segment $\vec 0p_2$, have lengths converging to
$2C$ and fluxes converging to $(0,2C,0)$. This construction of curves with non-zero
flux finishes the proof of item 3 of Theorem~\ref{geometry2}, assuming that the distance
between the lines $l_1$ and $l_2$ is $C$.

We can now prove that the distance $d$  between the lines $l_1$ and $l_2$ is $C$.
Arguing by contradiction, suppose that $d<C$ or $d>C$. If $d<C$ then, by the previous
arguments, there exists a sequence of loops $\Gamma_{n}$ containing the origin with the
norms of their fluxes bounded from below by $d$. Since the flux of a 1-cycle is a
homological invariant, this implies that such curves are homologically non-trivial.
Moreover the lengths of $\Gamma_{n}$ are converging to $2d<2C$. Therefore, there
exists $\ve>0$ such that for $n$ sufficiently large, $\Gamma_{n}\subset B_{M_n}(\vec 0, C-\ve)$. However, since
$\lim_{n\to \infty} [I_{{M_{n}}}(\vec{0})=C_n]=C$, for $n$ sufficiently
large $B_{M_n}(\vec 0, C-\ve)$ is a disk. This implies that for $n$ sufficiently
large, $\Gamma_{n}$ is homologically trivial which is a contradiction.

Suppose $d>C$.  Since
$\lim_{n\to \infty} [I_{M_n}(\vec 0)=C_n]=C\in (0,\infty)$,
there exists a sequence of simple closed geodesic loops
$\a_n\subset M_n$ based at $\vec 0$ and
of lengths $2C_n$ converging to $2C$ that are smooth everywhere
except possibly at $\vec 0$; see the proof of Claim~\ref{claim:Inj=Inf}. In fact,
arguing exactly as in the proof  of Claim~\ref{claim:Inj=Inf} gives that the limit
set of $\alpha_n$ must contain a point in $\cS-l_1$. Note that $\alpha_n\subset \ov \B(C_n)$.
Since $d>C$, there exists $\ve>0$ such that, for $n$ sufficiently large, $\alpha_n\subset \B(d-\ve)$.
In particular, for $n$ sufficiently large, $\alpha_n$ is at distance at least $\ve$
from the line $l_2=\cS-l_1$ and thus the limit set of $\alpha_n$ does not contain a
point in $\cS-l_1$. This contradiction proves that the distance between $l_1$ and $l_2$ is equal to $C$.

Finally, given $R>C$ let  $\Delta(n)$ be a connected component of $M_n\cap \B(R)$
that intersects $\B(\frac R4)$. We want to show that for $n$ sufficiently large, given
two points in $\Delta(n)$, their distance in $M_n$ is less than $3R$. Without loss of generality,
let us assume that $\Delta(n)$ is the connected component containing the origin.
Then, it suffices to show that given a point in $\Delta(n)$, its distance to the
origin in $M_n$ is less than $\frac32R$. Arguing by contradiction, assume there
exists $R>C$ and points $p(n)\in \Delta(n)$ at distance greater than or equal to
$\frac 32R$ to the origin. After going to a subsequence, let $p=\lim_{n\to\infty} p(n)$.
Recall that by Theorem~\ref{thm1.1} and Claim~\ref{claim5.5}, since the sequence
$\{M_n\}_{n\in\mathbb N}$ has locally positive injectivity radius in $\rth$, there
exists $\ve>0$ such that for $n$ sufficiently large, the intersection $\Delta(n)\cap\B(\ve)$
is a  disk that is contained in $B_{M_n}(R)$. Therefore, for $n$ sufficiently
large, $p_n\notin\B(\ve)$ which implies that $p\notin\B(\frac \ve2)$.

Let $\gamma$ be the horizontal line segment connecting $p$ to a point $q$ in the
$x_3$-axis and let $\alpha$ be the line segment in the $x_3$-axis connecting $q$
to the origin. If $\gamma\cap l_2\neq \O$,  let $z$ denote such point of intersection.
Note that the length of $\gamma\cup\alpha$ is less than $\sqrt 2R<\frac 32R$. By the
arguments used in the proof of this theorem, it is clear that there exists a sequence
of curves $\gamma(n)$ in $M_n$ connecting the origin to the point $p(n)$ that converges
to $\gamma\cup\alpha$ away from the points $q$ and $z$. Moreover, the lengths of this
curve converge to the length of $\gamma\cup\alpha$. Therefore, for $n$ sufficiently large,
the length of $\gamma(n)$ is less than $\frac 32R$ and so the distance from $p(n)$ to
the origin in $M_n$ is less than $\frac32R$. This contradiction proves that  for $n$
sufficiently large, given two points in $\Delta(n)$, their distance in $M_n$ is less than $3R$.

The geometric description given in item 2 of Theorem~\ref{geometry2} follows easily
from the arguments used in its proof. This finishes the proof of the theorem.

\begin{remark} \label{rem:6.6} {\em Suppose that for
some $\ve>0$, $\{M_n\}_{n\in \mathbb N}$  is a
sequence of compact $H_n$-surfaces in $\rth$ with finite genus
at most $k$,
 $\vec{0}\in {M_{n}}$,
$d_{{M_{n}}}(\vec{0},\partial {M(n)})\to \infty$ and that
$ I_{{M(n})}(x)\geq \ve$ for any $x\in {M(n)}$ with
$d_{M(n)}(x,\partial M(n))>1$. Then  Corollary~3.2 in~\cite{mt8}
shows that after replacing by a subsequence,
the components $M_n$ of $M(n)\cap\B(n)$ containing the origin
satisfy the conditions of one of the
Theorem~\ref{H-lam-thm}, \ref{geometry1} or \ref{geometry2}.
}
\end{remark}

\begin{definition}{\rm A point of {\em almost-minimal injectivity radius}\label{amininj}
of a compact surface $M$ surface  with boundary is a point  $p \in M$ where the function
$\frac{d_{M}(p,\partial M)}{I_{M}(p )}$  has
its maximal value.}
\end{definition}

As a consequence of  Remark~\ref{rem:6.6}, we have
the following proposition that is related to Theorem~1.1
in~\cite{mpr14}, which was proved under the hypothesis that $H=0$.

\begin{proposition} \label{cor:5.7}
Let $M(n)$ be a sequence of compact $H_n$-surfaces with boundary embedded in $\rth$ with
finite genus at most $k$ together with points $p_n\in M_n$ satisfying
$$\lim_{n\to\infty}\frac{d_{M(n)}(p_n,\partial M(n))}{I_{M(n)}(p_n)} = \infty.$$

Given points $q_n\in M(n)$  of almost-minimal injectivity
radius, there exist
 positive numbers $R_n$, $\lim_{n\to\infty}R_n= \infty$, satisfying:
\ben
\item  The component $M_n$ of $[\frac{1}{I_{M(n)}(q_n)} (M(n)-q_n)]\cap \B(R_n)$
containing $\vec{0}$ has its boundary in $ \partial \B(R_n)$ and genus at most $k$.
\item The sequence  $\{M_n\}_{n\in \N}$ has uniformly
positive injectivity radius in $\rth$ and
$I_{M_n}(\vec{0})=1$.\een
Then after choosing a subsequence and then
translating the surfaces $M_n$ by vectors
of length at most 1, the sequence $\{M_n\}_{n\in \N}$
satisfies the hypotheses of Theorems~\ref{geometry2}
with $C=1$ or  the sequence  $\{M_n\}_{n\in \N}$ satisfies
the hypotheses of Theorem~\ref{H-lam-thm}.
\end{proposition}
\begin{proof}
After choosing a subsequence suppose that
$$\frac{d_{M(n)}(p_n,\partial M(n))}{I_{M(n)}(p_n)}\geq n.$$
Let  $q_n\in M_n$ be points of almost-minimal injectivity radius and let
$M'_n$ be the scaled and translated $H'_n$-surface
$$M_n'=[\frac{1}{I_{M(n)}(q_n)} [B_{M(n)}(q_n,n/2)-q_n)].$$
Observe that $\{{M'_{n}}\}_{n\in \N}$  is a
sequence of compact $H_n$-surfaces in $\rth$ with genus
at most $k$,
 $\vec{0}\in {M'_{n}}$,
$\lim_{n\to \infty}d_{{M'_{n}}}(\vec{0},\partial {M'_n})$ and that
$ I_{M'_n}(x)\geq 1/2$ for any $x\in M'_n$ with
$d_{M'_n}(x,\partial M'_n)>1$.
Corollary~\ref{cor:5.7} now follows immediately from Remark~\ref{rem:6.6}.
\end{proof}

\section{Appendix.}

In this appendix we give the definition of a weak CMC
lamination of a Riemannian three-manifold. Specializing to the
case where all of the leaves have the same mean
curvature $H\in \R$, one obtains the definition of a
weak $H$-lamination, for which we give a few more explanations.
A simple example of a weak 1-lamination $\cL$ of $\rth$ that is
a not a 1-lamination is the union of two
spheres of radius 1 that intersect at single point of tangency.

For further background material on these notions see Section~3 of~\cite{mpr21},
\cite{mpr18}
or our previous papers~\cite{mt4,mt3}.

\begin{definition}
\label{definition}
{\rm A (codimension-one) {\it weak CMC lamination} ${\cal L}$ of a
Riemannian three-manifold $N$ is a collection
$\{ L_\alpha\}_{\alpha\in I}$ of (not necessarily
injectively) immersed constant mean curvature surfaces
called the {\it leaves} of ${\cal L}$,
satisfying the following four properties.
\begin{enumerate}[1.]
\item $\bigcup_{\alpha\in I}L_{\alpha }$ is a closed
subset of $N$. With an abuse of notation, we will also consider $\cL$
to be the closed set $\bigcup_{\alpha\in I}L_{\alpha }$.
\item The function $|A_{\cal L}|\colon {\cal L}\to [0,\infty )$ given by
\begin{equation}
\label{eq:sigma}
|A _{\cal L}|(p)=\sup \{ |A _L|(p)\ | \ L \mbox{ is a leaf of ${\cal L}$ with $p\in L$} \} .
\end{equation}
is uniformly bounded on compact sets of~$N$.
\item For every $p\in N$, there exists an $\ve_p>0$ such that if for some $\a\in I$,
$q\in L_\a \cap B_N(p,\ve_p)$, then $q$ contains a disk neighborhood in
$L_\a$ whose boundary is contained in $N- B_N(p,\ve_p)$.
\item If $p\in N$ is a point where either two leaves of
${\cal L}$ intersect or a leaf of ${\cal L}$ intersects itself, then
each of these  surfaces nearby $p$ lies at one side of the other
(this cannot happen if both of the intersecting leaves have the same
signed mean curvature as graphs over their common tangent space at
$p$, by the maximum principle).
\end{enumerate}
Furthermore:
\begin{itemize}
\item If $N=\bigcup _{\alpha } L_{\a }$, then we call ${\cal
L}$ a {\it weak CMC foliation} of $N$.

\item If the leaves of ${\cal L}$ have the same constant mean curvature
$H$, then we call ${\cal L}$ a {\it weak $H$-lamination} of $N$ (or
$H$-foliation, if additionally $N=\bigcup _{\alpha } L_{\a }$).
\end{itemize}
}
\end{definition}

\begin{figure}
\begin{center}
\includegraphics[width=5.1cm,height=4cm]{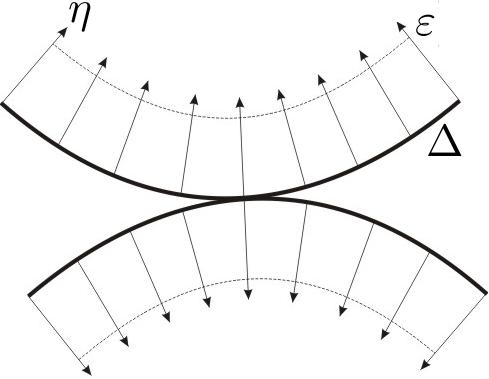}
\caption{The leaves of a weak $H$-lamination with $H\neq 0$ can
intersect each other or themselves, but only tangentially
with opposite mean curvature vectors. Nevertheless,
on the mean convex side of these locally intersecting leaves,
there is a lamination structure.
}
\label{figHlamin}
\end{center}
\end{figure}

The following proposition follows immediately from the definition of
a weak $H$-lamination and the maximum principle for $H$-surfaces.

\begin{proposition}
\label{prop10.2}
Any weak $H$-lamination
${\cal L}$ of a Riemannian three-manifold $N$ has a local $H$-lamination
structure on the mean convex side of each leaf.  More precisely,
given a leaf $L_{\a }$ of ${\cal L}$ and given a small disk $\Delta
\subset L_{\alpha }$, there exists an $\ve >0$ such that if $(q,t)$
denotes the normal coordinates for $\exp _q(t\eta _q)$ (here $\exp $
is the exponential map of $N$ and $\eta $ is the unit normal vector
field to $L_{\a }$ pointing to the mean convex side of $L_{\a }$),
then the exponential map $\exp $ is an injective submersion in
$U(\Delta ,\ve ):= \{ (q,t) \ | \ q\in \mbox{\rm Int}(\Delta ), \, t\in
(-\ve ,\ve )\} $, and the inverse image $\exp^{-1}({\cal L})\cap \{
q\in \mbox{\rm Int}(\Delta ), \,t\in [0,\ve )\} $ is an $H$-lamination of
$U(\Delta ,\ve $) in the pulled back metric, see
Figure~\ref{figHlamin}.
\end{proposition}

\begin{definition} \label{deflimit}{\rm
A leaf $L_\a$ of a weak $H$-lamination $\cL$ is a {\em limit leaf} of $\cL$
if at some $p\in L_\a$, on its mean convex side near $p$, it is a limit leaf of the
related local $H$-lamination given in Proposition~\ref{prop10.2}.
}\end{definition}

\begin{remark} \label{remarkweak}
   {\em
\begin{description}
\item[{\rm 1.}] A weak $H$-lamination for $H=0$ is a minimal lamination.
\item[{\rm 2.}] Every CMC lamination
(resp. CMC foliation) of a Riemannian three-manifold is a
weak CMC lamination (resp. weak CMC foliation).
\item[{\rm 3.}] Theorem~4.3 in~\cite{mpr19} states  that the
2-sided cover of a limit leaf of a
weak $H$-lamination is stable. By Lemma~3.3 in~\cite{mpr10}
and the main theorem in~\cite{ros9},
the only complete stable $H$-surfaces in $\rth$ are planes.
Hence, every leaf $L$ of a weak $H$-lamination $\cL$
of $\rth$ is properly immersed and has an embedded half-open
regular neighborhood $N(L)$ on its mean convex side,
and  $N(L)$ can be chosen to be disjoint from $\cL$ if $L$
is not a plane.  In particular, if $L$
is a leaf of a weak $H$-lamination $\cL$
of $\rth$, then there is a small perturbation $L'$ of $L$ in $N(L)$
that is properly embedded in $\rth$.

  \end{description}
  }
\end{remark}

\vspace{.3cm}
\center{William H. Meeks, III at profmeeks@gmail.com\\
Mathematics Department, University of Massachusetts, Amherst, MA 01003}
\center{Giuseppe Tinaglia at giuseppe.tinaglia@kcl.ac.uk\\ Department of
Mathematics, King's College London,
London, WC2R 2LS, U.K.}

\bibliographystyle{plain}
\bibliography{bill2}

\end{document}